\documentclass[a4paper,12pt, oneside]{amsart}
\usepackage{mathrsfs}
\usepackage{amsfonts}
\usepackage{amssymb}
\usepackage[polish, english]{babel}
\usepackage{amsxtra}
\usepackage{color}
\usepackage{dsfont}
\usepackage[margin=2cm, centering]{geometry}
\usepackage[bbgreekl]{mathbbol}
\usepackage{soul}

\usepackage[colorlinks,citecolor=blue,urlcolor=blue,bookmarks=true]{hyperref}

\DeclareSymbolFontAlphabet{\mathbb}{AMSb}
\DeclareSymbolFontAlphabet{\mathbbl}{bbold}

\usepackage{accents}

\theoremstyle{plain}
\newtheorem{theorem}{Theorem}[section] 
\newtheorem{lemma}[theorem]{Lemma}
\newtheorem{proposition}[theorem]{Proposition}
\newtheorem{corollary}[theorem]{Corollary}

 \theoremstyle{definition}

\newtheorem{example}{Example} 
\newtheorem*{remark*}{Remark} 
\newtheorem{remark}[theorem]{Remark}

\newcommand{\R}{\mathbb{R}}
\newcommand{\Rd}{{\R^{d}}}
\newcommand{\NN}{\mathbb{N}}
\newcommand{\ind}{\mathds{1}}
\renewcommand{\leq}{\leqslant}
\renewcommand{\geq}{\geqslant}
\newcommand{\Z}{\int^{\infty}_{0}}

\def \PP{\mathbb{P}}
\def \EE{\mathbb{E}}

\def\({\left(} 
\def\){\right)} 
\def\[{\left[}
\def\]{\right]} 
\def\<{\langle} 
\def\>{\rangle}

\def \Ca{\rm{(C1)}}
\def \Cb{\rm{(C2)}}
\def \Cc{\rm{(C3)}}
\def \Cd{\rm{(C4)}}
\def \Ce{\rm{(C5)}}
\def \Cf{\rm{(C6)}}
\def \Cg{\rm{(C7)}}
\def \Ch{\rm{(C8)}}

\def \CI{\rm{(CIm)}}

\def \Da{\rm{(D1)}}
\def \Db{\rm{(D2)}}
\def \Dc{\rm{(D3)}}
\def \Dd{\rm{(D4)}}

\def \Aa{\rm{(A1)}}
\def \Ab{\rm{(A2)}}
\def \Ac{\rm{(A4)}}
\def \Ad{\rm{(A3)}}
\def \Ae{\rm{(A5)}}

\def \Ba{\rm{(B1)}}
\def \Bb{\rm{(B2)}}
\def \Bc{\rm{(B4)}}
\def \Bd{\rm{(B3)}}
\def \Be{\rm{(B5)}}

\newcommand{\lah}{\alpha_h}

\newcommand{\LM}{N} 
\newcommand{\LCh}{\Psi} 
\newcommand{\drf}{b} 

\definecolor{ks}{rgb}{0.7,0.1,0.2}

\definecolor{zm}{RGB}{255,0,255}

\allowdisplaybreaks

\title{L{\'e}vy processes: concentration function and heat kernel bounds}
\thanks{The research was partially supported by
 the 
German Science Foundation (SFB 701)
and
National Science Centre (Poland)
grant 2016/23/B/ST1/01665.}

\author[T. Grzywny]{Tomasz Grzywny }
\address{
	Wydzia{\lll} Matematyki,
	Politechnika Wroc{\lll}awska\\
	Wyb. Wyspia\'{n}skiego 27\\
	50-370 Wroc{\lll}aw\\
	Poland
}
\email{tomasz.grzywny@pwr.edu.pl}

\author[K. Szczypkowski]{Karol Szczypkowski}
\email{karol.szczypkowski@pwr.edu.pl}



\date{}

\begin{document}


\def\1{{\bf 1}}
\def\ind{{\bf 1}}
\def\nn{\nonumber}
\newcommand{\I}{\mathbf{1}}

\def\sA {{\cal A}} \def\sB {{\cal B}} \def\sC {{\cal C}}
\def\sD {{\cal D}} \def\sE {{\cal E}} \def\sF {{\cal F}}
\def\sG {{\cal G}} \def\sH {{\cal H}} \def\sI {{\cal I}}
\def\sJ {{\cal J}} \def\sK {{\cal K}} \def\sL {{\cal L}}
\def\sM {{\cal M}} \def\sN {{\cal N}} \def\sO {{\cal O}}
\def\sP {{\cal P}} \def\sQ {{\cal Q}} \def\sR {{\cal R}}
\def\sS {{\cal S}} \def\sT {{\cal T}} \def\sU {{\cal U}}
\def\sV {{\cal V}} \def\sW {{\cal W}} \def\sX {{\cal X}}
\def\sY {{\cal Y}} \def\sZ {{\cal Z}}

\def\bA {{\mathbb A}} \def\bB {{\mathbb B}} \def\bC {{\mathbb C}}
\def\bD {{\mathbb D}} \def\bE {{\mathbb E}} \def\bF {{\mathbb F}}
\def\bG {{\mathbb G}} \def\bH {{\mathbb H}} \def\bI {{\mathbb I}}
\def\bJ {{\mathbb J}} \def\bK {{\mathbb K}} \def\bL {{\mathbb L}}
\def\bM {{\mathbb M}} \def\bN {{\mathbb N}} \def\bO {{\mathbb O}}
\def\bP {{\mathbb P}} \def\bQ {{\mathbb Q}} \def\bR {{\mathbb R}}
\def\bS {{\mathbb S}} \def\bT {{\mathbb T}} \def\bU {{\mathbb U}}
\def\bV {{\mathbb V}} \def\bW {{\mathbb W}} \def\bX {{\mathbb X}}
\def\bY {{\mathbb Y}} \def\bZ {{\mathbb Z}}
\def\R {{\mathbb R}} \def\RR {{\mathbb R}} \def\H {{\mathbb H}}
\def\n{{\bf n}} \def\Z {{\mathbb Z}}

\newcommand{\expr}[1]{\left( #1 \right)}
\newcommand{\cl}[1]{\overline{#1}}
\newtheorem{thm}{Theorem}[section]

\newtheorem{prop}[thm]{Proposition}

\numberwithin{equation}{section}
\def\ee{\varepsilon}
\def\qed{{\hfill $\Box$ \bigskip}}
\def\NN{{\mathcal N}}
\def\AA{{\mathcal A}}
\def\MM{{\mathcal M}}
\def\BB{{\mathcal B}}
\def\CC{{\mathcal C}}
\def\LL{{\mathcal L}}
\def\DD{{\mathcal D}}
\def\FF{{\mathcal F}}
\def\EE{{\mathcal E}}
\def\QQ{{\mathcal Q}}
\def\SS{{\mathcal S}}
\def\RR{{\mathbb R}}
\def\R{{\mathbb R}}
\def\L{{\bf L}}
\def\K{{\bf K}}
\def\S{{\bf S}}
\def\A{{\bf A}}
\def\E{{\mathbb E}}
\def\F{{\bf F}}
\def\P{{\mathbb P}}
\def\N{{\mathbb N}}
\def\eps{\varepsilon}
\def\wh{\widehat}
\def\wt{\widetilde}
\def\pf{\noindent{\bf Proof.} }
\def\pff{\noindent{\bf Proof} }
\def\cp{\mathrm{Cap}}

\maketitle

\begin{abstract}
We investigate 
densities of vaguely continuous convolution semigroups of probability measures on~$\Rd$.
We expose
that
many typical conditions on the characteristic exponent
 repeatedly used in the literature of the subject
are
equivalent to the behaviour
of the maximum of the density as a function of time variable.
We also prove
qualitative
lower 
estimates under mild
 assumptions on the corresponding jump measure and the
characteristic exponent. 
\end{abstract}

\section{Introduction}

Over the last years 
we
 observe a growing 
interest in studying analytic and probabilistic properties of L{\'e}vy processes.
It stems from a fact
that 
they
constitute a rich class of stochastic models which have many applications in finance, physics, biology and other fields.
The present paper is devoted
to a question of finding bounds to the transition density (the heat kernel) of a L{\'e}vy process.

We first briefly
 introduce the general framework
 and after that,
together with 
a few
 examples,
we describe our motivations.
Let
$d\in\N$ and $Y=(Y_t)_{t\geq 0}$ be
a L{\'e}vy process  in $\Rd$ (\cite{MR1739520}). 
Recall that there is a well known
one-to-one
 correspondence
between L{\'e}vy processes in $\Rd$ and
vaguely continuous convolution semigroups 
of probability measures $(P_t)_{t\geq 0}$ on $\Rd$.
Due to the presence of the convolution structure 
it is convenient to use
 Fourier transform in order to study~$Y$.
 Indeed, the celebrated L{\'e}vy-Khintchine formula says that
the characteristic exponent $\LCh$ of $Y$ defined by 
$$\mathbb{E}
e^{i\left<x,Y_t\right>}
=\int_{\Rd} e^{i\left<x,y\right>} P_t(dy)
=e^{-t\LCh(x)}\,,\qquad x\in\Rd\,,
$$ 
equals
$$
\LCh(x)=\left<x,Ax\right>-i\left<x,\drf\right> - \int_{\Rd} \left(e^{i\left<x,z\right>}-1 - i\left<x,z\right>\ind_{|z|<1}\right)\LM(dz)\,,
$$
where 
$A$ is a symmetric non-negative definite matrix, $\drf \in \Rd$ and
$\LM(dz)$ is a L{\'e}vy measure, i.e., 
a measure satisfying
$$
\LM(\{0\})=0\,,
\qquad\quad
\int_{\Rd} (1\land |z|^2)\LM(dz)<\infty \,.
$$
The triplet $(A,\LM,b)$ is called the generating triplet of $Y$. 
From that general perspective our
aim 
is to discuss the existence, and
even more, to establish 
certain
estimates of the transition
density $p(t,x)$ of~$Y_t$. Equivalently, 
it is a question of the absolute continuity of
$P_t(dx)$ with respect to the Lebesgue measure, and a problem of estimating
 its Radon-Nikodym derivative.
It is rather a standard practice to use the
characteristics describing continuous and jump part of a L{\'e}vy process
in order to formulate assumptions
and state results.
To this end 
for $r>0$ we define 
$$
h(r)=r^{-2} \|A\|+\int_{\Rd}\left(1\wedge\frac{|x|^2}{r^2}\right)\LM(dx)\,,
$$ 
and
$$
K(r)=r^{-2}\|A\|+r^{-2}\int_{|x|<r}|x|^2\LM(dx)\,.
$$
The function $h$ 
is called the {\it concentration function}.
It is significant
from the point of view of analysis and probability.
We comment on that in a few lines.
Note that $|e^{-t \LCh(x)}|=e^{-t {\rm Re}[\LCh(x)]}$ and if $e^{-t \LCh(x)}$
is absolutely integrable,
then we can invert the Fourier transform
and represent the transition density as follows,
\begin{align*}
p(t,x)=(2\pi)^{-d}\int_{\R^d} e^{-i\left<x,z\right>}e^{-t\LCh (z)}\, dz\,.
\end{align*} 
Readily, the real part of $\LCh$
equals
$
{\rm Re}[\LCh(x)]=\left<x,Ax\right>+\int_{\Rd}\big( 1-\cos \left<x,z\right> \big)\LM(dz)
$. Next we consider its radial, continuous and non-decreasing majorant
defined by
$$
\LCh^*(r)=\sup_{|z|\leq r} {\rm Re}[\LCh(z)],\qquad r>0\,.
$$
From \cite[Lemma~4]{MR3225805} we  have 
\begin{align}\label{ineq:comp_TJ}
\frac{1}{8(1+2d)} h(1/r)\leq \LCh^*(r) \leq 2 h(1/r)\,,\qquad r>0\,.
\end{align}
Thus $h$ is a more tractable version 
of $\LCh^*$. See Lemma~\ref{lem:basic_prop_K_h} for basic properties of $h$.
On the other hand,
there exists a constant $c>0$, depending only on the dimension $d$,
such that (see~\cite{MR632968})
$$ c^{-1}/h(r) \leq \E[ S (r)] \leq c/h(r)\,,\qquad r>0\,,$$
where $S (r)=\inf \{t\colon |Y_t-t \drf_r|> r \}$ and
\begin{align}\label{def:br}
\drf_r=\drf+\int_{\Rd} z \left(\ind_{|z|<r} - \ind_{|z|<1}\right)\LM(dz)\,.
\end{align}
Intuitively, $h$ describes the  average expansion of the process in the space.
For other
results relating $h$ to probabilistic quantities of L{\'e}vy processes see for instance \cite{MR3350043}.

A  natural question 
is
whether 
the function $h$
may also be used to control the distribution of the process, that is
the transition density $p(t,x)$. 
Among many examples for which this is the case one reports
the Wiener process
and  isotropic
$\alpha$-stable processes $\alpha\in (0,2)$.
Before giving a precise formulation 
let us note that
these are
two types of L{\'e}vy processes that exhibit radically different behaviour
on the level of  realizations -- continuous/c{\`a}ld{\`a}g trajectories --
and in terms of the decay rate of the
transition density at infinity -- exponential/power-type decay.
Namely,
if we
denote by $g(t,x)$ and $p_{\alpha}(t,x)$ the 
corresponding
transition densities,
we have that
for all $t>0$ and $x\in\Rd$ (see \cite{MR0119247} and \cite{MR2274838}),
\begin{align*}
g(t,x)=(2\pi t)^{-d/2} e^{-\frac{|x|^2}{2t}}\,,
\qquad\quad\qquad p_{\alpha}(t,x)\approx \min\left\{ t^{-d/\alpha},\, t|x|^{-d-\alpha}\right\}.
\end{align*}
By $f\approx g$ we mean
that the quotient $f/g$ is bounded between to positive constants.
Despite the differences, 
these
processes
share 
certain common or at least similar properties.
Their transition densities can be expressed by the inverse Fourier transform
with the
respective
characteristic exponents $|x|^2$ and $|x|^{\alpha}$,
the corresponding functions $h(r)$ are
up to multiplicative constants equal to
$r^{-2}$ and $r^{-\alpha}$, 
while the inverse $h^{-1}$ evaluated at $1/t$ is $t^{1/2}$ 
and $t^{1/\alpha}$, respectively.
Further, for all $t>0$,
$$
\sup_{x\in\Rd}
g(t,x)=g(t,0) = c t^{-d/2}\,,\qquad\quad  \sup_{x\in\Rd} p_{\alpha}(t,x)=p(t,0)= c t^{-d/\alpha}\,.
$$
The above equalities, understood
as 
inequalities
 "$\leq$",
are known 
as the on-diagonal upper bounds, and they are crucial in the theory of
symmetric processes on metric measure spaces 
\cite{MR2492992},
\cite{MR2285739},
\cite{MR898496},
\cite{MR1141283},
\cite{MR1232845}
as well as on $\Rd$
\cite{MR2995789},
\cite{MR2886383}.
They may further lead to near- and off-diagonal bounds 
when accompanied by additional assumptions
\cite{MR1010819}.
Putting aside this context,
we observe that
the transition densities of the Wiener process and isotropic $\alpha$-stable processes satisfy
\begin{align}\label{initial}
\sup_{x\in\Rd} p(t,x)\leq c [h^{-1}(1/t)]^{-d}\,,
\end{align}
which yields the desired control 
by $h$.
The validity of \eqref{initial}
for a given
L{\'e}vy process
is
the principal 
subject
of our study.
In this connection,
in Section~\ref{sec:equivalence}
we 
consecutively
reveal 
numerous
 descriptions of \eqref{initial},
which are expressed
via conditions that relate
the transition density $p$,
the characteristic exponent  $\LCh$ 
and functions $\LCh^*$, $h$ and $K$.
Many of them are 
derived from
the literature
where they 
typically
serve as a starting point
for further investigation of particular subclasses of L{\'e}vy processes. Therefore the equivalences we obtain
not only
enhance
the
comprehension of \eqref{initial} itself,
but also
provide a
clarification
of the existing results 
and
enable
significant 
reduction
of assumptions
(\cite{MR3235175}, \cite{MR3357585}, \cite{MR3666815}, \cite{MR3604626}).
In particular, we propose the following characterisation  which
exposes two key features that describe L{\'e}vy processes satisfying~\eqref{initial}. Roughly these are scaling and comparability of projections.

{\it A L{\'e}vy process  in $\Rd$ has a transition density $p(t,x)$ satisfying \eqref{initial} for all $t\in (0,1]$
and some fixed constant $c>0$
if and only if
the average expansion given by $h(r)$ fulfils 
certain
weak scaling condition at zero, 
and each of the projections of the process 
on a one-dimensional subspace of $\Rd$
locally expands in the same manner as the original process,
moreover this comparability should be uniform under the choice of the projection.}

A rigorous formulation of this result may be found in
Lemma~\ref{lem:C8}.
We
note
that the description becomes 
more transparent 
if $d=1$, since
any projection equals the original process,
the scaling turns to be the determining feature
(see Remark~\ref{rem:d=1}).
For example, any $\alpha$-stable process with $\alpha\in (0,2)$ in one dimension satisfies \eqref{initial}.
In particular, $\alpha$-stable subordinators $\alpha\in (0,1)$ 
constitute an example for which \eqref{initial} holds.
These are one-dimensional L{\'e}vy processes which lack any symmetry as their distributions are supported on the right half-line.
Therefore, 
even though
the two previously discussed examples are
rotationally invariant (hence symmetric) 
unimodal L{\'e}vy processes \cite[Definition~14.12 and~23.2]{MR1739520},
neither
the invariance (or symmetry) nor the unimodality  is necessary 
for~\eqref{initial}.
It is also known that they are not
sufficient.
For instance, in
\cite{TGBTMR} the authors considered
such processes with transition densities satisfying
$$
\sup_{x\in\Rd} p(t,x)=p(t,0)=\infty,\qquad t\in (0,1).
$$
However, 
if a L{\'e}vy process is rotationally invariant, 
a similar to the one dimensional phenomenon occurs,
and \eqref{initial} becomes
equivalent to the scaling
(see Remark~\ref{rem:rot_inv}, cf. \cite[Proposition~19,  Corollary~20]{MR3165234}).
For other positive examples we refer the reader for instance to
\cite{MR2443765},
\cite{MR3646773},
\cite{MR1085177},
\cite{MR1287843}, 
\cite{MR2025727},
\cite{MR2307406},
\cite{MR3089797},
\cite{MR3010850},
\cite{MR941977},
\cite{MR2591907},
\cite{MR705619},
\cite{MR854867}.
We emphasise that with the results of the present paper it is easier to
classify which of the L{\'e}vy processes discussed in the literature
fall into the class satisfying \eqref{initial}.

We will now
show that 
\eqref{initial}
may fail for a decent symmetric process.
Let $X^{\alpha_1}$, $X^{\alpha_2}$, $X^{\alpha_3}$ be independent one-dimensional symmetric stable
processes with $\alpha_1, \alpha_2, \alpha_3 \in (0,2)$
and  consider
$Y_t=(X^{\alpha_1}_t,X^{\alpha_2}_t, X^{\alpha_3}_t)$.
The transition density of $Y_t$ equals $$p(t,x)=p_{\alpha_1}(t,x_1)p_{\alpha_2}(t,x_2)p_{\alpha_3}(t,x_3)\,,$$
where $x=(x_1,x_2,x_3)\in\RR^3$.
Consequently,
$$
\sup_{x\in\RR^3} p (t,x)= p(t,0)= c t^{-1/\alpha_1 - 1/\alpha_2-1/\alpha_3}\,,\qquad t>0\,,
$$
while
$h$ is comparable with $r^{-\max\{\alpha_1, \alpha_2, \alpha_3\}}$
for $r\in (0,1)$ and with $r^{-\min\{\alpha_1, \alpha_2,\alpha_3\}}$ if $r\geq 1$.
Thus, if $\alpha_1<\alpha_2<\alpha_3$,
the quantity $[h^{-1}(1/t)]^{-d}$
does not provide an upper bound for
$\sup_{x\in\RR^3} p (t,x)$.
In such case 
projections of $Y$
on the coordinate axes
have average expansions that do not compare.
The function $h$
that measures the expansion of the original 
process over balls
does not 
detect such nuances 
in the behaviour 
and hence it does not carry necessary information to control
the distribution.
More sensitive
but perhaps also much more complicated
objects than~$h$,
like those proposed in
\cite{MR2925579},
would have to be introduced 
to include this kind of examples into the discussion.
This is beyond the scope of that paper.

Finally, the results of Section~\ref{sec:equivalence}
show  that \eqref{initial}
is related to lower estimates. In particular, it implies one
of a form
$$
p(t,x+\Theta)\geq c \left[h^{-1}(1/t)\right]^{-d}\,,
$$
for a specific range of $t>0$, $x\in\Rd$
and a proper choice of a shift $\Theta\in\Rd$.  
The aforementioned result of \cite{MR632968} relating the average expansion with $h$ suggests that $\Theta$ should   
incorporate
the quantity \eqref{def:br}
to grasp the internal shift of the process caused by the constant drift $\drf$ and the non-symmetry of the L{\'e}vy measure $\LM(dz)$.
It appears that $\Theta$ should also sense where the maximum of the density is attained.
More extensive discussion is pursued at the beginning
of Section~\ref{sebsec:low_C3}.
Recall that a L{\'e}vy process is symmetric
if and only if $\drf=0$ and $\LM(dz)$ is a symmetric measure,
and then if the transition density exists it attains its maximum at the origin.
This substantially facilitates
the 
analysis
for symmetric L{\'e}vy processes.
Qualitative 
results for non-symmetric once are
less present in the literature,
mostly 
performed in a generality that
allows only rather implicit estimates
(\cite{MR3139314}, \cite{MR3235175}, \cite{MR3357585})
or carried out for
very peculiar cases (\cite{MR0282413}, \cite{MR1486930}, \cite{MR3626900}, \cite{MR2794975}).

We
note
that $h(0^+)<\infty$ ($h$ is bounded) if and only if $A=0$ and $\LM(\Rd)<\infty$, i.e., the corresponding L\'{e}vy process is a compound Poisson process (with drift). 
Most of the conditions discussed in the paper automatically
preclude $Y$ from being 
such a process.
Nevertheless,
to avoid unnecessary considerations
{\bf we assume in the whole paper that
$h(0^+)=\infty$}.

The remainder of the paper is organized as follows. In Section~\ref{sec:K&h} we collect fundamental properties of functions $K$ and $h$.
In Section~\ref{sec:equivalence} we prove the equivalence of several conditions 
for small time
and 
separately for large time.
In Section~\ref{sec:decomposition}
we propose an auxiliary decomposition of a L{\'e}vy process.
Section~\ref{sebsec:low_C3} is dedicated to the lower estimates of the transition denisty.
Examples and further applications are given in Section~\ref{sec:ex_ap}.

We conclude this section by 
setting
the notation.
Throughout the article
$\omega_d=2\pi^{d/2}/\Gamma(d/2)$ is the surface measure of the unit sphere in $\R^d$.
$B_r$ is a ball of radius $r$ centred at the origin.
By $c(d,\ldots)$ we denote a generic
 positive constant that depends only on the listed parameters $d,\ldots$. 
We write $f(x)\approx g(x)$, or simply $f\approx g$,
if there is a constant $c\in [1,\infty)$ independent of $x$ such that
$c^{-1} f(x)\leq g(x)\leq c f(x)$.
As usual $a\land b=\min\{a,b\}$ and $a\vee b = \max\{a,b\}$.
In some proofs we use a short notation of the
weak lower scaling condition (at infinity), i.e., for
$\phi\colon (0,\infty)\to [0,\infty]$ we say that
$\phi$ satisfies ${\rm WLSC}(\underline{\alpha},\underline{\theta},\underline{c})$
or $\phi \in {\rm WLSC}(\underline{\alpha},\underline{\theta},\underline{c})$
if there are $\underline{\alpha}\in\RR$, $\underline{\theta}\geq 0$ and $\underline{c}\in (0,1]$ such that
$$
\phi(\lambda r) \geq \underline{c} \lambda^{\underline{\alpha}} \phi(r)\,,\qquad \lambda \geq 1\,, r>\underline{\theta}\,.
$$
Borel sets in $\Rd$ will be denoted by $\mathcal{B}(\Rd)$.
A Borel measure $\nu$ on $\Rd$ is called symmetric if
$\nu(A)=\nu(-A)$ for every $A\in \mathcal{B}(\Rd)$.

\section*{Acknowledgment}
The authors thank A. Bendikov, K. Bogdan, A.~Grigor'yan, S. Molchanov, R. Schilling and P. Sztonyk
for helpful comments.

\section{Preliminaries - functions $K$ and $h$}\label{sec:K&h}

In this section we discuss a L{\'e}vy process $Y$ in $\Rd$
with a generating triplet $(A,\LM,\drf)$.
The following properties 
are often used without further comment.

\begin{lemma}\label{lem:basic_prop_K_h}
We have
\begin{enumerate}
\item[\rm 1.] $\lim_{r\to\infty}h(r)=\lim_{r\to\infty}K(r)=0$,
\item[\rm 2.] $h$ is continuous and strictly decreasing,
\item[\rm 3.] $r^2 K(r)$ and $r^2h(r)$ are non-decreasing,
\item[\rm 4.] $\lambda^2 K(\lambda r) \leq K(r)$ 
and $\lambda^2 h(\lambda r)\leq h(r)$, $\lambda\leq 1$, $r>0$,
\item[\rm 5.] $\sqrt{\lambda} h^{-1}(\lambda u)\leq h^{-1}(u)$, $\lambda\geq 1$, $u>0$.
\item[\rm 6.] For all $r>0$,
\begin{align*}
&\int_{|z|\geq r }  \LM(dz)\leq  h(r)
\quad \mbox{and} \quad
\int_{|z|< r}  |z|^2 \LM(dz) \leq  r^2 h(r)\,.
\end{align*}
\end{enumerate}
\end{lemma}
\pf
The first property follows from the dominated convergence theorem and $K \leq h$. Similarly we get the continuity of $h$. Next, since we assume that $h(0^+)=\infty$, we get either that $\|A\|\neq 0$ or $\LM(\Rd)=\infty$ (hence for every $l>0$ there is $0<k<l$ such that $\int_{k<|x|<l}\LM(dx)>0$). Each of them guarantees that $h$ decreases in a strict sense.
The remaining parts follow easily from the definition of $K$ and $h$.
\qed

\begin{lemma}\label{rem:int_K_is_h}
For all $0<a<b\leq \infty$ we have
$$
h(b)-h(a)=-\int_a^b 2 K(r)r^{-1}\,dr\,.
$$
\end{lemma}
\pf
It suffices to consider the non-local part for $a>0$ and $b=\infty$. By Fubini's theorem
\begin{align*}
\int_a^\infty 2 r^{-3} \int_{|x|<r} |x|^2 \LM(dx)dr
= \int_{\Rd}|x|^2 \int_{a\vee |x|}^{\infty} 2 r^{-3} dr \LM(dx)= \int_{\Rd}|x|^2 
(a\vee |x|)^{-2}\LM(dx)=h(a)\,.
\end{align*}
\qed

\begin{lemma}\label{lem:equiv_scal_h}
Let $\lah \in(0,2]$, $C_h\in[1,\infty)$ and $\theta_h\in(0,\infty]$. 
The following are equivalent.

\begin{enumerate}
\item[\Aa] For all $\lambda\leq 1$ and $r< \theta_h$,
\begin{equation*}
 h(r)\leq C_h\lambda^{\lah }h(\lambda r)\,. 
\end{equation*}
\item[\Ab]  For all $\lambda\geq 1$ and $u>h(\theta_h)$,
\begin{equation*}
 h^{-1}(u)\leq (C_h\lambda)^{1/\lah}\, h^{-1}(\lambda u)\,. 
\end{equation*}
\end{enumerate}
\noindent
Further, consider
\begin{enumerate}
\item[\Ad] There is $\underline{c}\in (0,1]$ such that
for all $\lambda \geq 1$ and $r>1/\theta_h$,
\begin{align*}
\LCh^* (\lambda r) \geq \underline{c} \lambda^{\lah} \LCh^*(r)\,.
\end{align*}
\item[\Ac] There is $c>0$ such that for all $r<\theta_h$,
\begin{equation*}
h(r)\leq c K(r)\,.
\end{equation*}
\item[\Ae]
There are $c>0$ and $\theta\in (0,\infty]$ such that for all $\lambda\leq 1$ and $r< \theta$,
\begin{equation*}
 K(r)\leq c \lambda^{\lah }K(\lambda r)\,. 
\end{equation*}
\end{enumerate}
Then,   
$\Aa$
gives $\Ad$
with $\underline{c}= 1/(c_d C_h)$,   $c_d=16(1+2d)$, while
$\Ad$
gives 
$\Aa$
with $C_h=c_d/\underline{c}$.\\
$\Aa$ 
 implies $\Ac$  
with $c=c(\lah,C_h)$.
$\Ac$
implies 
$\Aa$
with $\lah=2/c$ and $C_h=1$.
$\Aa$ gives $\Ae$ with $c=c(\lah,C_h)$ and $\theta=\theta_h$. 
$\Ae$ implies $\Aa$ with $C_h=c$ and $\theta_h=h^{-1}(2h(\theta))$.
\end{lemma}
\pf
 We show that $\Ab$ 
 gives $\Aa$.
The converse implication is proved in the same manner.
Let $u=h(r)$. Then $r<\theta_h$ is the same as $u>h(\theta_h)$. 
If $\lambda \in (0,C_h^{-1/\lah})$ we let $s=(C_h \lambda^{\lah})^{-1} \geq 1$
and by $\Ab$ 
we get
$$
h(\lambda r)=h( (C_hs)^{-1/\lah} h^{-1}(u)) \geq s u = (C_h \lambda^{\lah})^{-1} h(r)\,.
$$
If $\lambda \in [C_h^{-1/\lah},1]$, then  $(C_h \lambda^{\lah})^{-1}\leq 1 $  and by the monotonicity
of $h$,
$$
h(\lambda r) \geq h(r) \geq (C_h \lambda^{\lah})^{-1} h(r)\,.
$$
The equivalence of $\Aa$
and $\Ad$ 
follows from \eqref{ineq:comp_TJ}.
We show the equivalence of $\Aa$ 
and $\Ac$.
By $\Aa$ 
 we have
$h(s)\leq \frac{1}{2}h(\lambda_0 s)$
for $s<\theta_h$ and $\lambda_0=1/(2C_h)^{1/\lah}<1$.
By Lemma~\ref{rem:int_K_is_h},
\begin{equation*}
K(s)\geq \frac{2}{\lambda_0^{-2}-1}\int^s_{\lambda_0 s}r^2K(r)\frac{dr}{r^3}=\frac{1}{\lambda_0^{-2}-1}(h(\lambda_0 s)-h(s))\geq \frac{1/2}{\lambda_0^{-2}-1}h(\lambda_0 s)\geq \frac{1}{\lambda_0^{-2}-1}h( s).
\end{equation*}
Conversely, again by Lemma~\ref{rem:int_K_is_h} we get for $0<r_1<r_2<\theta_h$,
\begin{equation*}
h(r_2)-h(r_1)\leq - (2/c) \int_{r_1}^{r_2} h(s)s^{-1}\,ds\,,
\end{equation*}
which implies that $h(r) r^{2/c}$ is non-increasing for $r<\theta_h$, and ends this part of the proof.
From $\Aa$ we get $\Ae$ by using $\Ac$. Now, if we assume $\Ae$, then
for $\lambda\leq 1$ and $r<h^{-1}(2h(\theta))$,
\begin{align*}
\frac1{2}h(r)= h(r)-h(\theta)= \int_r^{\theta}K(s)s^{-1}ds
&\leq 
c \lambda^{\lah} \int_r^{\theta} K(\lambda s)s^{-1}ds\\
&\leq c\lambda^{\lah} \int_{\lambda r}^{\lambda \theta} K(u)u^{-1}du
\leq c \lambda^{\lah} h(\lambda r)\,.
\end{align*}
This ends the proof.
\qed

\begin{lemma}\label{lem:2_impl_scal}
Assume that for some $T,c_1,c_2>0$ we have
$$
\int_{\Rd} e^{-c_1 t\, {\rm Re}[\LCh(z)]} dz \leq c_2 \left[h^{-1}(1/t)\right]^{-d}\,,\qquad t<T\,.
$$
Then 
$\Aa$
  holds for some $\lah \in (0,2]$, $C_h\in [1,\infty)$ and $\theta_h=h^{-1}(1/T)$.
Moreover, $\lah$ and $C_h$ can be chosen to depend only on $d$, $c_1$ and $c_2$.
\end{lemma}
\pf
By \eqref{ineq:comp_TJ}
\begin{align*}
\int_{\Rd} e^{-c_1t\, {\rm Re}[\LCh(z)]} dz
&\geq \int_{|z|< 1/h^{-1}(2/t)} e^{-c_12t \,h(1/|z|)}dz
\geq e^{-c_1 2t \,h(h^{-1}(2/t))} \omega_d \left[h^{-1}(2/t)\right]^{-d}\\
&= e^{-4c_1} \omega_d \left[h^{-1}(2/t)\right]^{-d}\,.
\end{align*}
Thus for $c_0=(c_2 e^{4c_1}/\omega_d)^{1/d}$ we have 
$h^{-1}(1/t) \leq c_0 h^{-1}(2/t)$, $t<T$.
Letting $c=\max\{c_0, \sqrt{2}\}$, $\sigma =\log_2 (c)$ and
considering  $2^{n-1} \leq \lambda<2^n$, $n\in\N$, we get for $t<T$,
$$
h^{-1}(1/t)\leq c \lambda^{\sigma} h^{-1}(\lambda/t)\,.
$$
The statement follows from Lemma~\ref{lem:equiv_scal_h}.
\qed

Note that in Lemma~\ref{lem:equiv_scal_h}
and~\ref{lem:2_impl_scal} we 
deal with the behaviour of
the function $h$
at the origin (or globally if $\theta_h=\infty$ therein).
Without proofs we give counterparts  
for the behaviour at infinity.

\begin{lemma}\label{lem:equiv_scal_h-2}
Let $\lah \in(0,2]$, $c_h\in(0,1]$ and $\theta_h\in[0,\infty)$. 
The following are equivalent.

\begin{enumerate}
\item[\Ba] For all $\lambda\geq  1$ and $r> \theta_h$,
\begin{equation*}
c_h\lambda^{\lah } h( \lambda r)\leq h(r)\,. 
\end{equation*}
\item[\Bb]  For all $\lambda\leq 1$ and $u<h(\theta_h)$,
\begin{equation*}
(c_h\lambda)^{1/\lah} h^{-1}(\lambda u)\leq  h^{-1}(u)\,. 
\end{equation*}
\end{enumerate}
\noindent
Further, consider
\begin{enumerate}
\item[\Bd] There is $\overline{c}\in [1,\infty)$ such that for all $\lambda \leq 1$ and $r<1/\theta_h$,
\begin{equation*}
\LCh^*(\lambda r) \leq \overline{c} \lambda^{\lah} \LCh^*(r)\,. 
\end{equation*}
\item[\Bc] There is $c>0$ and $\theta\in [0,\infty)$ such that for all $r>\theta$,
\begin{equation*}
h(r)\leq c K(r)\,. 
\end{equation*}
\item[\Be]
There are $c>0$ and $\theta\in [0,\infty)$ such that for all $\lambda\geq 1$ and $r> \theta$,
\begin{equation*}
c \lambda^{\lah} K(\lambda r)\leq K(r)\,. 
\end{equation*}

\end{enumerate}
Then,   
$\Ba$
gives $\Bd$
with $\overline{c}= c_d/ c_h$,   $c_d=16(1+2d)$, while
$\Bd$
gives 
$\Ba$
with $c_h=1/(c_d \overline{c})$.\\
$\Ba$ 
 implies $\Bc$  
with $c=c(\lah,c_h)$ and $\theta=(c_h/2)^{-1/\lah}\theta_{h}$.
$\Bc$
implies 
$\Ba$
with $\lah=2/c$, $c_h=1$ and $\theta_h=\theta$.
$\Ba$ gives $\Be$ with $c=c(\lah,c_h)$ and $\theta=(c_h/2)^{-1/\lah}\theta_{h}$.
$\Be$ implies $\Ba$ with $c_h=c$ and $\theta_h=\theta$.
\end{lemma}

\begin{lemma}\label{lem:2_impl_scal-2}
Assume that for some $T,c_1,c_2>0$ we have
$$
\int_{\Rd} e^{-c_1 t\, {\rm Re}[\LCh(z)]} dz \leq c_2 \left[h^{-1}(1/t)\right]^{-d}\,,\qquad t>T\,.
$$
Then 
$\Ba$
  holds for some $\lah \in (0,2]$, $c_h\in (0,1]$ and $\theta_h=h^{-1}(2/T)$.
Moreover, $\lah$ and $c_h$ can be chosen to depend only on $d$, $c_1$ and $c_2$.
\end{lemma}

Here are a few more general formulae that relate
other objects to $\int_{|z|\geq r}\LM(dz)=\LM(B_r^c)$.
\begin{lemma}
Let $f\colon [0,\infty)\to [0,\infty)$ be differentiable, $f(0)=0$, $f'\geq 0$ and $f'\in L^1_{loc}([0,\infty))$.
For all $r>0$,
\begin{align}
\int_{|z|<r} f(|z|)\,\LM(dz)
&=\int_0^r f'(s)
\LM(B_s^c)\,ds
-f(r)
N(B_r^c)
,\label{eq:tail_1}\\
\int_{|z|\geq r}f(|z|)\, \LM(dz)
&=\int_0^{\infty} f'(s) 
\LM(B_{r\vee s}^c)\,ds\,. \label{eq:tail_2}
\end{align}
\end{lemma}
\pf We have
\eqref{eq:tail_1} by
\begin{align*}
\int_{|z|<r} f(|z|)\,\LM(dz)
&=\int_{\Rd} \ind_{|z|<r} \left( \int_0^{\infty}\ind_{s\leq |z|} f'(s)\,ds\right)\LM(dz)\\
&=\int_0^r f'(s) \left( \int_{\Rd} \ind_{s\leq |z|<r} \,\LM(dz)  \right)ds\\
&= \int_0^r f'(s) \left( \int_{\Rd} \ind_{s \leq |z|} \,\LM(dz)  \right)ds- \int_0^r f'(s) \left(\int_{|z|\geq r}\LM(dz)\right)ds\,.
\end{align*}
The equality \eqref{eq:tail_2}
follows from
\begin{align*}
\int_{|z|\geq r}f(|z|)\, \LM(dz)
=\int_{\Rd} \ind_{r \leq |z|} \left(\int_0^{\infty} \ind_{s\leq |z|} f'(s)\,ds\right)\LM(dz)\,.
\end{align*}
\qed

Putting $f(s)=s^2$ in \eqref{eq:tail_1}
gives the following formula.
\begin{corollary}\label{cor:h_rep}
For all $r>0$,
$$
h(r)=r^{-2} \|A\|+ r^{-2} \int_0^r 2s
\,\LM(B_s^c)
\,ds\,.
$$
\end{corollary}

\begin{lemma}\label{lem:int_infty}
Let $\Aa$ hold with $\lah\geq 1$. If $A=0$, then $\int_{|z|<1}|z|\LM(dz)=\infty$.
\end{lemma}
\pf
By \eqref{eq:tail_1} with $f(s)=s$ we have
$
\int_{|z|<1}|z|\LM(dz)=\int_0^1 \LM(B_s^c)\,ds-\LM(B_1^c)
$.
By Corollary~\ref{cor:h_rep} we get
$
rh(r)\leq 2 \int_0^r \LM(B_s^c)\,ds
$. By our assumption the left hand side of the latter is bounded from below by a positive constant,
so $\int_0^r \LM(B_s^c)\,ds=\infty$
and the proof is complete.
\qed

\begin{lemma}\label{lem:lah_g1_1}
Let $\Aa$ hold with $\lah>1$. Then
$$\int_{r\leq |z|< \theta_h}|z| \LM(dz)\leq  \frac{2 C_h}{\lah-1}\,  r h(r)\,,\qquad r>0\,.$$
\end{lemma}
\pf 
By \eqref{eq:tail_2} with $f(s)=s$ and
the L{\'e}vy measure $\ind_{|z|<\theta_h}\LM(dz)$,
\begin{align*}
\int_{r \leq |z|<\theta_h} |z|\,\LM(dz)
&=
\int_0^{\theta_h} \int_{|z|\geq r\vee s}\LM(dz)\,ds
\leq
\int_0^{\theta_h} h(r\vee s)\,ds\\
&\leq r h(r)+\int_r^{\theta_h} h(s)\,ds
\leq r h(r)+ \int_r^{\theta_h} C_h (r/s)^{\lah} h(r)\,ds\,.
\end{align*}
\qed

\begin{corollary}\label{cor:drf_r_drf}
Let $\Aa$ hold with $\lah>1$. Then there is
a constant $c=c(d,\lah,C_h)$ such that
for all $0<r<\theta_h$,
\begin{align*}
|\drf_r-\drf|\leq \frac{c}{\theta_h \land 1} \max\left\{r,r^2\right\}\, h(r)\,.
\end{align*}
\end{corollary}
\begin{proof}
If $r\geq 1$, then $|\drf_r-\drf|\leq r^2h(r)$.
Let $r\leq 1$. We have 
\begin{align*}
|\drf_r-\drf| 
&\leq 
\int_{r\leq |z|<1} |z|\,\LM(dz)
\leq
\int_{r\leq |z|<\theta_h} |z|\, \LM(dz)
+ \int_{|z|\geq \theta_h \land 1} \LM(dz)\,.
\end{align*}
By $\Aa$ we get
\begin{align*}
\int_{|z|\geq \theta_h \land 1} \LM(dz)
\leq h(\theta_h \land 1) 
\leq C_h(r/(\theta_h\land 1))\, h(r)\,,
\end{align*}
which ends the proof by Lemma~\ref{lem:lah_g1_1}.
\end{proof}

We end this section with a technical comment on $\Aa$ and $\Ba$.
\begin{remark}\label{rem:rozciaganie}
If $\theta_{h}<\infty$ in $\Aa$,
we can stretch the range of scaling to  $r <R<\infty$
at the expense of the constant $C_{h}$.
Indeed, by continuity of $h$,  for $\theta_h\leq r< R$,
$$
h(r)\leq h(\theta_h)\leq C_h \lambda^{\lah} h(\lambda \theta_h)\leq C_h (r/\theta_h)^2 \lambda^{\lah} h(\lambda r)\leq C_h (R/\theta_h)^2 \lambda^{\lah}  h(\lambda r)\,.
$$
Similarly, if $\theta_h>0$ in $\Ba$, we extend the range to $0<R<r$
by reducing the constant $c_h$.
We have for $R<r\leq \theta_h$, 
$$
h(r)\geq h(\theta_h)\geq c_h \lambda^{\lah} h(\lambda \theta_h) 
\geq c_h (r/\theta_h)^2 \lambda^{\lah} h(\lambda r)
\geq c_h (R/\theta_h)^2 \lambda^{\lah} h(\lambda r) \,.
$$
\end{remark}

\section{General L{\'e}vy processes
}\label{sec:equivalence}

In this section we discuss a L{\'e}vy process $Y$ in $\Rd$
with a generating triplet $(A,\LM,\drf)$.

\subsection{Equivalent conditions - small  time}
We introduce and comment on eight conditions $\Ca-\Ch$, 
which are common in the literature. 
For $\Cb$ and $\Ce$ see \cite{MR2995789, MR3357585, MR3604626}, for $\Cc$ see \cite{MR3165234},
and for $\Cd$ see \cite{MR3139314, MR3235175}.

\begin{theorem}\label{thm:equiv}
Let $Y$ be a L{\'e}vy process.
The following are equivalent.

\begin{enumerate}
\item[\Ca] The density $p(t,x)$ of $Y_t$ exists and
there are $T_1\in(0,\infty]$, $c_1>0$ such that for all $t<T_1$,
$$
\sup_{x\in\Rd} p(t,x) \leq c_1 \left[h^{-1}(1/t)\right]^{-d}.
$$
\item[\Cb] There are $T_2\in (0,\infty]$, $c_2>0$ such that for all $t<T_2$,
$$
\int_{\Rd} e^{-t\, {\rm Re}[\LCh(z)]} dz \leq c_2 \left[h^{-1}(1/t)\right]^{-d}.
$$
\item[\Cc]  There are $T_3\in (0,\infty]$, $c_3\in (0,1]$ and $\alpha_3\in (0,2]$ such that for all $|x|>1/T_3$,

$$c_3\,  \LCh^*(|x|)\leq {\rm Re}[\LCh(x)]\qquad \mbox{and}
\qquad 
\LCh^*(\lambda r)\geq c_3 \lambda^{\alpha_3} \LCh^*(r)\,,\quad \lambda \geq 1,\,r>1/T_3\,.
$$ 
{ \quad}
\item[\Cd] There are $T_4\in(0,\infty]$, $c_4\in [1,\infty)$ such that for all $|x|>1/T_4$,
$$
\LCh^*(|x|)
\leq c_4\left( \left<x,Ax\right>+\int_{|\left<x,z\right>|<1} |\left<x,z\right>|^2\LM(dz)\right) \,.
$$
\end{enumerate}
Moreover, if $T_i=\infty$ for some $i=1,\ldots,4$, then $T_i=\infty$ for all $i=1,\ldots,4$.
\end{theorem}

\pf
$\Cb \implies  \Ca$.
Follows immediately by the inverse Fourier transform.

\noindent
$\Ca \implies  \Cb$. Note that $p(t/2,\cdot)\in L^{1}(\Rd)\cap L^{\infty}(\Rd)\subset L^{2}(\Rd)$ for every $t>0$. Thus $e^{-(t/2)\LCh(\cdot)}\in L^{2}(\Rd)$ or equivalently
$|e^{-(t/2)\LCh(\cdot)}|^2
=e^{- t {\rm Re}[\LCh(\cdot)]}\in L^1(\Rd)$.
In particular, 
$p(t,\cdot)\in C_0(\Rd)$ holds by the Riemann--Lebesgue lemma.
Now, let $Z=Y^1-Y^2$, where $Y^1$ and $Y^2$ are two indepndent copies of $Y$.
Then $Z$ has  $2 {\rm Re}[\LCh(x)]$ as the characteristic exponent and a density $p_Z(t,\cdot)\in C_0(\Rd)$ such that for all $x\in\Rd$,
$$
p_Z(t,x)= \int_{\Rd} p(t,x-y)p(t,y)\,dy= (2\pi)^{-d} \int_{\Rd} e^{-i\left<x,z\right>} e^{-2 t \,{\rm Re}[\LCh(z)]}\, dz\,.
$$
Consequently, we get for $t<T_1$
$$
\int_{\Rd} e^{-(2 t)\, {\rm Re}[\LCh(z)]}\, dz\leq c_1 \left[h^{-1}(1/t)\right]^{-d}= c_1 \left[h^{-1}(2/(2t))\right]^{-d}\,,
$$
and the statement follows by Lemma~\ref{lem:2_impl_scal} and~\ref{lem:equiv_scal_h}  with 
$c_2=c_2(d,c_1)$ and $T_2=T_1/2$.

\noindent
$\Cb \implies  \Cd$.
The case of $d=1$ is simpler and
follows from Lemma~\ref{lem:2_impl_scal}, $\Ac$ and 
\eqref{ineq:comp_TJ}.
 We focus on $d\geq 2$.
For $x \neq 0$ let $v=x /|x|$ and $\Pi_1 z=\left<v,z\right>v$ be a projection on the linear subspace $V=\{\lambda v \colon \lambda \in \RR\}$ of $\Rd$.
We consider a projection $Z^1=\Pi_1 Y$ of the L{\'e}vy process $Y$ on $V$
and the corresponding objects $\LCh_1$, $K_1$ and~$h_1$.
By~\cite[Proposition~11.10]{MR1739520},
\begin{align*}
\LCh_1(z)=&\LCh(\Pi_1 z)\,,\quad z\in\Rd\,,\\
K_1(r)=r^{-2}\|\Pi_1 A \Pi_1\|&+r^{-2}\int_{|\Pi_1 z|< r}  |\Pi_1 z|^2 \LM(dz)\,,\\
h_1(r)=r^{-2}\|\Pi_1 A \Pi_1\|&+ \int_{\Rd}\left(1\wedge\frac{|\Pi_1 z|^2}{r^2}\right)\LM(dz)\,.
\end{align*}
Note that
$$
K_1(1/|x|)=\left<x,Ax\right>+\int_{|\left<x,z\right>|<1} |\left<x,z\right>|^2\LM(dz)\,.
$$
Therefore it suffices to show that for all $r<T_4$ (see \eqref{ineq:comp_TJ}),
\begin{align}\label{eq:crucial}
2 \,h(r) \leq c_4 K_1(r)\,,
\end{align}
with $c_4>0$ independent of the choice of $x$, or equivalently of the choice of the projection~$\Pi_1$.
Similarly, we define $Z^2=\Pi_2 Y$ and we get $\LCh_2$, $K_2$ and $h_2$ for a projection $\Pi_2$ on the linear subspace $V^{\perp}=\{y\in\Rd \colon  \left<y,v\right>=0\}$.
We let $\{v,v_2,\ldots,v_d\}$ to
be an orthonormal basis (with the usual scalar product) such that $v_2,\ldots,v_d\in V^{\perp}$.
Then
$x=
\xi v + \xi_2 v_2+ \ldots +\xi_d v_d$,
where $\xi\in\R$, $\bar{\xi}=(\xi_2,\ldots,\xi_d)\in \R^{d-1}$,
and we write $x=(\xi,\bar{\xi})$.
Since ${\rm Re}[\LCh(x)]$ is a characteristic exponent we have by \cite[Proposition~7.15]{MR0481057} 
that
\begin{align*}
\sqrt{{\rm Re}[\LCh(\xi,\bar{\xi})]} \leq 
\sqrt{{\rm Re}[\LCh(\xi,0)]}
+\sqrt{{\rm Re}[\LCh(0,\bar{\xi})]}
= \sqrt{{\rm Re}[\LCh_1(\xi,0)]}
+ \sqrt{{\rm Re}[\LCh_2(0,\bar{\xi})]}
\,.
\end{align*}
Thus ${\rm Re}[\LCh(\xi,\bar{\xi})]
\leq 
2 {\rm Re}[\LCh_1(\xi,0)]
+ 2{\rm Re}[\LCh_2(0,\bar{\xi})]$.
In particuliar, see \eqref{eq:unbounded_LCh}, both $\LCh_1$ and $\LCh_2$ are unbounded,
so $Z^1$ and $Z^2$ are not compound Poisson processes (with drift), therefore
$h_1$ and $h_2$ are unbounded and strictly decreasing.
Further, by \eqref{ineq:comp_TJ}
for $t<T_2$,
\begin{align}
c_2 \left[h^{-1}(1/t)\right]^{-d}&\geq \int_{\Rd} e^{-t\, {\rm Re}[\LCh(z)]} dz\geq
\left(\int_{\R} e^{-2t\, {\rm Re}[\LCh_1( \xi,0)]}d\xi\right) \left( \int_{\R^{d-1}}e^{-2t\,{\rm Re}[\LCh_2(0,\bar{\xi})]}  d\bar{\xi}\right)  \label{eq:unbounded_LCh} \\
&\geq \left(\int_{|\xi|<1/h_1^{-1}(1/t)} e^{-4t\, h_1(1/|\xi|)} d\xi \right) \left(\int_{|\bar{\xi}|<1/h_2^{-1}(1/t)} e^{-4t\, h_2(1/|\bar{\xi}|) }d\bar{\xi}\right) \nonumber \\
&\geq e^{-8} \omega_{d-1} \left[h_1^{-1}(1/t) \right]^{-1} \left[h_2^{-1}(1/t)\right]^{-(d-1)} \,.\nonumber
\end{align}
Directly from the definition we have $h_2\leq h$, which 
implies $h^{-1}_2 \leq h^{-1}$ and with the above
gives 
$$
h^{-1}(u) \leq c_0 \, h_1^{-1}(u)\,,\qquad u>1/T_2\,,
$$
with $c_0=\max\{1,(c_2 e^8 /\omega_{d-1})\}$.
This implies by
monotonicity of $r^2 h_1(r)$ that
$$
h(r)\leq h_1(r/c_0)\leq c_0^2\, h_1(r)\,, \qquad r<h^{-1}(1/T_2)\,.
$$
By 
Lemma~\ref{lem:2_impl_scal} 
$h$  satisfies $\Aa$
with some $\lah=\lah(d,c_2)$, $C_h=C_h(d,c_2)$
and $\theta_h=h^{-1}(1/T_2)$.
Consequently, since $h_1$ and $h$ are comparable ($h_1\leq h$ always holds),
$h_1$ satisfies $\Aa$   
 with $\lah$, $c_0^2 C_h$ and $\theta_h$.
Lemma~\ref{lem:equiv_scal_h} for $h_1$ assures 
\eqref{eq:crucial} with
$c_4=c_4(d,c_2)$
and
$T_4=h^{-1}(1/T_2)$.

\noindent
$\Cd \implies  \Cc$.
Note that $1-\cos(r) \geq (1-\cos(1))r^2$ for $|r|<1$. 
Thus, together
with the assumption  we have 
for $|x|>1/T_4$,
\begin{align*}
{\rm Re}[\LCh(x)]
\geq  \left<x,Ax\right>
+(1 -\cos (1)) \int_{|\left<x,z\right>|<1} |\left<x,z\right>|^2 \LM(dz)
\geq \frac{1-\cos(1)}{c_4}\, \LCh^*(|x|)\,.
\end{align*}
It remains to show that $\LCh^*\in {\rm WLSC}$, or equivalently that
$\Aa$ holds
for $h$.
We take $v\in\Rd$ such that $|v|=1$
and we let $\Pi_1$ to be a projection on
the linear subspace $V=\{\lambda v \colon \lambda \in \RR\}$ of~$\Rd$.
We consider a projection $Z^1=\Pi_1 Y$ of the L{\'e}vy process $Y$ on $V$
and the corresponding objects $K_1$ and~$h_1$.
Note that for $r>0$,
$$
K_1(r)=\left<(v/r),A(v/r)\right>+\int_{|\left<(v/r),z\right>|<1} |\left<(v/r),z\right>|^2 \LM(dz)\,.
$$
and therefore by \eqref{ineq:comp_TJ} and our assumption for $r<T_4$,
\begin{align*}
h_1 (r)\leq h(r) \leq c_4 8(1+2d) K_1(r) \leq c_4 8(1+2d) h_1(r)\,.
\end{align*}
Using Lemma~\ref{lem:equiv_scal_h} we get 
$\Aa$
for~$h_1$
with $\alpha_{h_1}=\alpha_{h_1}(d,c_4)$, $C_{h_1}=1$
and $\theta_{h_1}=T_4$. 
Since $h_1$ and $h$
are comparable we conclude
$\Aa$
 for~$h$.
Finally, the result holds with
$\alpha_3=\alpha_3(d,c_4)$, $c_3=c_3(d,c_4)$
and $T_3=T_4$.

\noindent
$\Cc \implies \Cb$.
By \eqref{ineq:comp_TJ} and our assumption ${\rm Re}[\LCh(x)]\geq c[ h(1/|x|)-h(T_3)]$
for all $x\in\Rd$ with $c=c(d,c_3)\leq 1$. Next,
by Lemma~\ref{lem:equiv_scal_h}
$\Aa$
holds with $\lah=\alpha_3$, $\theta_h=T_3$
and $C_h=c_d/c_3$, $c_d=16(1+2d)$.
In particular,
$h^{-1}(1/(ct))\geq (c c_3/c_d)^{1/\alpha_3} h^{-1}(1/t)$
for $t<1/h(T_3)$.
Further, $h(1/r)$ is increasing and satisfies ${\rm WLSC}(\alpha_3 ,1/T_3, c_3/c_d)$.
Then by \cite[Lemma~16]{MR3165234}  for $t<1/h(T_3)$,
\begin{align*}
\int_{\Rd} e^{-t\, {\rm Re}[\LCh(z)]} dz \leq 
e^{ct h(T_3)}
\int_{\Rd} e^{-ct\, h(1/|z|)} dz
\leq 
C
e^{ct h(T_3)}
\left[h^{-1}(1/(ct))\right]^{-d}
\leq c_2 \left[h^{-1}(1/t)\right]^{-d}.
\end{align*}
To sum up, $\Cb$ holds with $c_2=c_2(d,\alpha_3,c_3)$ and $T_2=1/h(T_3)$.
\qed

\begin{remark}\label{rem:d=1}
If $d=1$ the conditions $\Ca-\Cd$ are tantamount to conditions 
$\Aa-\Ac$.
Indeed, in such case $\Cd$ reduces to $\Ac$
with $\theta_h=T_4$ and $c$ related to $c_4$ according to
\eqref{ineq:comp_TJ}.
\end{remark}

\begin{remark}\label{rem:rot_inv}
If $Y$ is rotationally invatiant (see \cite[Definition~14.12]{MR1739520}),
then the conditions $\Ca-\Cd$ are tantamount to conditions 
$\Aa-\Ac$. 
In particular,
$\Cd$ lightens to $\Ac$.
\end{remark}
\noindent
We give a short justifications.
Plainly,
$\Cc$ implies $\Ad$.
On the other hand,
by
\cite[Exercise~18.3]{MR1739520}
we have
\begin{align*}
\left<x,Ax\right>+\int_{|\left<x,z\right>|<1} |\left<x,z\right>|^2\LM(dz)
&= a |x|^2+|x|^2 \int_{|z_i|<1/|x|} |z_i|^2 \LM(dz)\\
&\geq a |x|^2+|x|^2 \int_{|z|<1/|x|} |z_i|^2 \LM(dz)\,,
\qquad i=1,\ldots,d\,.
\end{align*}
Thus $\Ac$ and \eqref{ineq:comp_TJ} give
exactly $\Cd$ by
$$
\LCh^*(|x|)\leq 2 h(1/|x|)\leq 2 c K(1/|x|)\leq  2cd \left(\left<x,Ax\right>+\int_{|\left<x,z\right>|<1} |\left<x,z\right>|^2\LM(dz)\right).
$$

From the next result we see that $\Cb$ implies 
bounds for higher moments, i.e.,
bounds for the spatial derivatives of the density.

\begin{proposition}
The conditions of Theorem~\ref{thm:equiv} are equivalent with
\begin{enumerate}
\item[\Ce] There is $T_5\in(0,\infty]$ such that
for some (every) $m\in\N$ there is $c_5>0$ and for all $t<T_5$,
$$
\int_{\Rd} |z|^m e^{-t\, {\rm Re}[\LCh(z)]}\, dz \leq c_5 \left[h^{-1}(1/t)\right]^{-d-m}.
$$
\end{enumerate}
Moreover, $\Cc$ implies $\Ce$ with $c_5=c_5(d,m,\alpha_3,c_3)$ and $T_5=1/h(T_3)$.
\end{proposition}
\pf
First we show that $\Cc$ gives $\Ce$ for every $m\in\N$.
By \eqref{ineq:comp_TJ} and our assumption there is $c=c(d,c_3)\leq 1$
such that for all $t>0$,
\begin{align*}
\int_{\Rd}|z|^m e^{-t\, {\rm Re}\left[\LCh(z)\right]}\,dz 
&
\leq 
e^{ct h(T_3)}
\int_{\Rd\setminus \{0\}} |z|^m e^{- c t\,  h(1/|z|)} \,dz
= e^{ct h(T_3)}\omega_d \int_0^{\infty} e^{- c t\,  h(1/r)} r^{m+d-1} \,dr\\
&= e^{ct h(T_3)} \frac{\omega_d}{\omega_{m+d}} \int_{\R^{m+d}\setminus \{0\}} e^{-c t\, h(1/|\xi|)}\,d\xi\,.
\end{align*}
Let $c_d=16(1+2d)$. By Lemma~\ref{lem:equiv_scal_h}
$h(1/r)$ satisfies ${\rm WLSC}(\alpha_3 ,1/T_3, c_3/c_d)$
and $h^{-1}(1/(ct))\geq (c c_3/c_d)^{1/\alpha_3} h^{-1}(1/t)$
for $t<1/h(T_3)$.
By 
\cite[Lemma~16]{MR3165234}
for all $t<1/h(T_3)$,
\begin{align*}
\int_{\R^{m+d}\setminus \{0\}} e^{-c t\,  h(1/|\xi|)}\,d\xi
\leq C \big[ h^{-1}(1/(ct)) \big]^{-d-m}
\leq c_5 \big[ h^{-1}(1/t) \big]^{-d-m}\,.
\end{align*}
Here $c_5=c_5(d,m,\alpha_3,c_3)$.
It remains to prove that if $\Ce$ holds for some $m\in\N$, then $\Cb$ also holds.
Indeed, $\Cb$ follows by
\begin{align*}
\int_{\Rd} e^{-t\, {\rm Re}[\LCh(z)]}\, dz
\leq \int_{|z|\leq 1/h^{-1}(1/t)}dz+  \left[h^{-1}(1/t)\right]^m \int_{|z|>1/h^{-1}(1/t)}|z|^m e^{-t\, {\rm Re}[\LCh(z)]}\, dz.
\end{align*}
\qed

Observe that
for all $r_1,r_2>0$ we have
\begin{align}\label{ineq:dif_br}
|\drf_{r_1}-\drf_{r_2}|\leq \int_{r_1\land r_2\leq |z| < r_1\vee r_2} |z| \LM(dz)
\leq (r_1\vee r_2)h(r_1\land r_2) \,.
\end{align}

\begin{lemma}\label{lem:CI}
The conditions of Theorem~\ref{thm:equiv} imply that
\begin{enumerate}
\item[\CI] 
The density $p(t,x)$ of $Y_t$ exists and
there are $T\in (0,\infty]$, $c\in [1,\infty)$ such that for every
$t<T$ there exists 
$|x_t|\leq c h^{-1}(1/t)$
so that for every $|y|\leq (1/c)h^{-1}(1/t)$,
$$
p(t,y+x_t+t\drf_{[h^{-1}(1/t)]})\geq (1/c) \left[h^{-1}(1/t)\right]^{-d}\,.
$$
\end{enumerate}
Moreover, $\Cc$ implies  $\CI$ with $c=c(d,\alpha_3,c_3)$ and $T=1/h(T_3/c)$.
If $T_3<\infty$
in $\Cc$,
then $\CI$ holds for every $T>0$ with $c=c(d,\alpha_3,c_3,T_3,T,h)$.
\end{lemma}
\begin{proof}
We note that
there is $a_0=a_0(d,\alpha_3,c_3)\geq 1$ such that
for 
$\lambda:=a_0 h^{-1}(1/t)<T_3$
we have
$
\PP(|Y_t-t \drf_{\lambda}|\geq \lambda)\leq 1/2
$.
Indeed, by
\cite{MR632968}
there is $c=c(d)$ such that for
$r=\lambda$,
\begin{align*}
\PP(|Y_t -t \drf_{\lambda}|\geq r) 
&\leq c t  \left( r^{-1} \left| (\drf-\drf_{\lambda})+\int_{\Rd} z \left(\ind_{|z|<r} - \ind_{|z|<1}\right)\LM(dz)\right|  +h(r)\right)
= ct  h(r)\,,
\end{align*}
and applying Lemma~\ref{lem:equiv_scal_h} we get
$
h(r)=h(\lambda)\leq (c_d/c_3) a_0^{-\alpha_3} h(\lambda/a_0)=(c_d/c_3) a_0^{-\alpha_3} t^{-1}
$.
Then
\begin{align}\label{step:sym_rem}
1/2 \leq 1-\PP(|Y_t - t \drf_{\lambda}|\geq \lambda) 
= \int_{|x-t \drf_{\lambda}|< \lambda } p(t,x)\,dx
\leq  
\omega_d \,
\lambda^d
 \sup_{|x|< \lambda } \big[ p(t,x+t \drf_{\lambda}) \big].
\end{align}
Therefore, by the continuity of $p$,
whenever $\lambda<T_3$, 
then there exists 
$|\xi_t|\leq \lambda$ such that
$p(t, \xi_t+t \drf_{\lambda} ) \geq 1/(2\omega_d)\, \lambda^{-d}$.
Further, by $\Ce$ there is $c_5=c_5(d,\alpha_3,c_3)$ such that
$
\sup_{x\in\Rd} |\nabla_x p(t,x)| \leq c_5 /(2\omega_d) \, \lambda^{-d-1} 
$
for every $t<1/h(T_3)$.
This gives  for $\lambda<T_3$ and $|y|\leq 1/(2 c_5)\, \lambda$,
$$
p(t,\xi_t+t \drf_{\lambda}+y)\geq p(t,\xi_t+t\drf_{\lambda})-|y| \sup_{x\in\Rd} |\nabla_x p(t,x)|\geq 1/(4 \omega_d)  \, \lambda^{-d}.
$$
Finally, for
 every $t<1/h(T_3/a_0)$,
$x_t=\xi_t+t(\drf_{\lambda}-\drf_{[h^{-1}(1/t)]})$
and every $|y|\leq a_0/(2 c_5)\,  h^{-1}(1/t)$,
\begin{align*}
p(t,x_t +t\drf_{[h^{-1}(1/t)]}+y)=p(t,\xi_t + t \drf_{\lambda}+y)\geq 
1/(4 \omega_d) \left[ a_0  h^{-1}(1/t)\right]^{-d}\,.
\end{align*}
Note  that $|x_t|\leq 2 a_0 h^{-1}(1/t)$,
because 
by \eqref{ineq:dif_br} we have
$
t |\drf_{\lambda}-\drf_{[h^{-1}(1/t)]}|
\leq \lambda$.
Now we prove the last sentence of
the statement.
It suffices to show that if $\CI$
hods with $T>0$ and $c\geq 1$,
then it also holds with $2T$ and a modified $c$,
where the modificaton depends only on $d,\alpha_3,c_3,T_3,T,h$.
Let $t<2T$ and $x_t=2 x_{t/2}-t\drf_{[h^{-1}(1/t)]}+t\drf_{[h^{-1}(2/t)]}$. Then
by Chapman-Kolmogorov equation,
\begin{align*}
&p(t,y+x_t+t\drf_{[h^{-1}(1/t)]})\\
&\geq  \int_{|z|<(1/c) h^{-1}(2/t)} 
p(t/2,y-z+x_{t/2}+(t/2)\drf_{[h^{-1}(2/t)]})\,
p(t/2,z+x_{t/2}+(t/2)\drf_{[h^{-1}(2/t)]})\,dz\\
&\geq \int_{|z|<(1/c) h^{-1}(2/t)} 
p(t/2,y-z+x_{t/2}+(t/2)\drf_{[h^{-1}(2/t)]})\,dz
\,(1/c) \left[h^{-1}(2/t)\right]^{-d}.
\end{align*}
By Lemma~\ref{lem:equiv_scal_h}
and the monotonicity of $h^{-1}$
there is $\tilde{c}=\tilde{c}(\alpha_3,c_3,T_3,T,h)$
such that
$h^{-1}(u)\leq \tilde{c} h^{-1}(2u)$, $u>1/(2T)$.
Then for $|y|\leq 1/(2c \tilde{c}) \,h^{-1}(1/t)$
and $|z|< 1/(2c) \,h^{-1}(2/t)$
we have
$|y-z|\leq (1/c)h^{-1}(2/t)$,
thus
\begin{align*}
\int_{|z|<(1/c) h^{-1}(2/t)} 
p(t/2,y-z+x_{t/2}+(t/2)\drf_{[h^{-1}(2/t)]})\,dz
\geq  (1/c) \omega_d (2c)^{-d}.
\end{align*}
Note that $|x_t|\leq 2 (c+1)h^{-1}(1/t)$
 by 
the bound of $|x_{t/2}|$
and
\eqref{ineq:dif_br}.
The proof is complete.
\end{proof}

Here are two consequences
of
merging
 Lemma~\ref{lem:CI}
with the condition $\Ca$
(note that $\Cf$ implies $\Ca$ by integrating over a ball 
of radius $(1/c_6)h^{-1}(1/t)$).

\begin{corollary}\label{lem:C6}
The conditions of Theorem~\ref{thm:equiv} are equivalent with
\begin{enumerate}
\item[\Cf] 
The density $p(t,x)$ of $Y_t$ exists and
there are $T_6\in (0,\infty]$, $c_6\in [1,\infty)$ such that for every
$t<T_6$ there exists 
$|x_t|\leq c_6 h^{-1}(1/t)$
so that for every $|y|\leq (1/c_6)h^{-1}(1/t)$,
$$
p(t,y+x_t+t\drf_{[h^{-1}(1/t)]})\geq (1/c_6) \sup_{x\in\Rd} p(t,x)\,.
$$
\end{enumerate}
Moreover, $\Cc$ implies  $\Cf$ with $c_6=c_6(d,\alpha_3,c_3)$ and $T_6=1/h(T_3/c_6)$.
If $T_3<\infty$
in $\Cc$,
then $\Cf$ holds for every $T_6>0$ with $c_6=c_6(d,\alpha_3,c_3,T_3,T_6,h)$.
\end{corollary}

The next corollary,
which is in the spirit of $\Ca$,
gives another connection with
the existing literature,
cf.  \cite[Theorem~2.1]{MR3139314}.

\begin{corollary}\label{cor:C7}
The conditions of Theorem~\ref{thm:equiv} are equivalent with
\begin{enumerate}
\item[\Cg]
The density  $p(t,x)$ of $Y_t$ exists and there are $T_7\in (0,\infty]$, $c_7\in [1,\infty)$
such that for all $t<T_7$,
$$
c_7^{-1} \left[h^{-1}(1/t)\right]^{-d}\leq \sup_{x\in\Rd} p(t,x)  \leq c_7 \left[h^{-1}(1/t)\right]^{-d}\,.
$$
\end{enumerate}
Moreover, $\Cc$ implies $\Cg$ with 
$c_7=c_7(d,\alpha_3,c_3)$ and $T_7=1/h(T_3/c_7)$.
If $T_3<\infty$
in $\Cc$,
then $\Cg$ holds for every $T_7>0$ with $c_7=c_7(d,\alpha_3,c_3,T_3,T_7,h)$.
\end{corollary}

We 
elucidate
a crucial difference between a general (possibly non-symmetric) case and
the situation when $\drf=0$ and $\LM(dz)$ is symmetric.
\begin{remark}\label{rem:sym}
If $Y$ is a symmetric L{\'e}vy process we have $\drf_r=0$ for all $r>0$ and
moreover we can take $x_t=0$ in the statements of
Lemma~\ref{lem:CI} and
Corollary~\ref{lem:C6}. Therefore the two results
provide a lower (near-diagonal) bound for $p(t,y)$.
Indeed, in the proof of
\eqref{step:sym_rem} we have
$$
 \sup_{|x|< \lambda } \big[ p(t,x+t \drf_{\lambda}) \big]=p(t,0)
$$
and we may take $\xi_t=0$ and thus also $x_t=0$.
\end{remark}

There are at least several ways how to reformulate  the condition $\Cc$, only using 
\eqref{ineq:comp_TJ}
and Lemma~\ref{lem:equiv_scal_h},
to discover more about its meaning.
We will present one such reformulation which formalizes
the description of \eqref{initial} presented in the introduction. 
\begin{lemma}\label{lem:C8}
The conditions of Theorem~\ref{thm:equiv} are equivalent with
\begin{enumerate}
\item[\Ch] 
There are $T_8\in (0,\infty]$, $c_8\in [1,\infty)$ and $\alpha_8\in (0,2]$ such that for 
every projection $\Pi_1$ on a one-dimensional subspace of $\Rd$,
$$ h(r)\leq c_8\, h_1(r)\qquad \mbox{and}
\qquad 
h(r)\leq c_8 \,\lambda^{\alpha_8} h(\lambda r)
\,,\quad \lambda \leq  1,\,r<T_8\,.
$$
where $h_1$ corresponds to a projected L{\'e}vy process $\Pi_1 Y$.
\end{enumerate}
\end{lemma}
\pf
Note that we always have $h_1\leq h$, since 
$
h_1(r)=r^{-2}\|\Pi_1 A \Pi_1\|+ \int_{\Rd}(1\land \frac{|\Pi_1 z|^2}{r^2})\LM(dz)
$
\cite[Proposition~11.10]{MR1739520}.
We first prove $\Ch\implies \Cc$.
Due to Lemma~\ref{lem:equiv_scal_h}
it suffices to
focus on the first part of $\Cc$.
Let $x\in\Rd$, $x\neq 0$,
and
consider $\Pi_1$ to be a projection on a subspace spanned by $v=x/|x|$.
Since $h$ and $h_1$ are comparable on $r<T_8$
we get
$\Ac$ for $h_1$, which together with
\eqref{ineq:comp_TJ} gives
for $|x|>1/T_8$,
\begin{align*}
\LCh^*(|x|)&\leq 2 h(1/|x|)\leq 2c_8\, h_1(1/|x|)
\leq 2 c_8\, c(\alpha_8,c_8)\, K_1(1/|x|)\\
&= 2 c_8\, c(\alpha_8,c_8)\left( 
\left<x,Ax\right>+\int_{|\left<x,z\right>|<1} |\!\left<x,z\right>\!|^2\,\LM(dz)\right)
\leq 2 c_8 c(\alpha_8,c_8) {\rm Re}[\LCh(x)]\,.
\end{align*}
Thus $\Cc$ holds with $c_3=c_3(d,\alpha_8,c_8)$, $T_3=T_8$, $\alpha_3=\alpha_8$.
Now we establish $\Cc\implies\Ch$.
Let $v\in\Rd$, $|v|=1$, be such that $\Pi_1$ projects on a subspace spanned by $v$.
We denote by $\LCh_1$ the characteristic exponent of $\Pi_1 Y$. Recall that $\LCh_1(z)=\LCh(\Pi_1 z)$.
Then for $r<T_8$ we set $x=rv$ to get
\begin{align*}
c_3 \LCh^*(r)\leq {\rm Re}[\LCh(x)]
= {\rm Re}[\LCh_1(x)]\leq \LCh_1^*(r)\,,
\end{align*}
which by \eqref{ineq:comp_TJ} proves $\Ch$ with
$c_8=c_8(d,c_3)$, $T_8=T_3$ and $\alpha_8=\alpha_3$.
\qed

\subsection{Equivalent conditions - large time}

Our next result
resembles 
Theorem~\ref{thm:equiv}, 
except that here we analyse
the density for large time.
The main difference is 
that in the third and the fourth 
condition below we add a priori
that 
from some point in time onwards
 the characteristic function
is absolutely integrable.

\begin{theorem}\label{thm:equiv-2}
Let $Y$ be a L{\'e}vy process.
The following are equivalent.
\begin{enumerate}
\item[\Da] 
There are $T_1, c_1>0$ such that 
the density $p(t,x)$ of $Y_t$ exists for all $t>T_1$ and
$$
\sup_{x\in\Rd} p(t,x) \leq c_1 \left[h^{-1}(1/t)\right]^{-d}.
$$
\item[\Db] There are $T_2, c_2>0$ such that for all $t>T_2$,
$$
\int_{\Rd} e^{-t\, {\rm Re}[\LCh(z)]} dz \leq c_2 \left[h^{-1}(1/t)\right]^{-d}.
$$
\item[\Dc] 
There are $T_3>0$, $c_3\in (0,1]$ and $\alpha_3\in (0,2]$ such that for all $|x|<1/T_3$,
$$c_3\,  \LCh^*(|x|)\leq {\rm Re}[\LCh(x)]\qquad \mbox{and}\qquad 
\LCh^*(\lambda r) \leq (1/c_3) \lambda^{\alpha_3} \LCh^*(r), \quad \lambda \leq 1,\,r<1/T_3\,.
$$ 
We have  $e^{-t_0 \LCh}\in L^1(\Rd)$ for some $t_0>0$.\\

\item[\Dd] There are  $T_4>0$, $c_4\in [1,\infty)$ such that for all $|x|<1/T_4$,
$$
\LCh^*(|x|)
\leq c_4\left( \left<x,Ax\right>+\int_{|\left<x,z\right>|<1} |\left<x,z\right>|^2\LM(dz)\right) \,.
$$
We have  $e^{-t_0 \LCh}\in L^1(\Rd)$ for some $t_0>0$.
\end{enumerate}
\end{theorem}

\begin{proof}
$\Db \implies \Da$ is direct. $\Da \implies \Db$ with $c_2=c_2(d,c_1)$ and $T_2=4T_1$,
$\Db \implies \Dd$
with $c_4=c_4(d,c_2)$ and $T_4=c(d,c_2) h^{-1}(1/T_2)$,
and $\Dd \implies \Dc$ with $\alpha_3=\alpha_3(d,c_4)$, $c_3=c_3(d,c_4)$ and $T_3=T_4$, 
by proofs similar to that of Theorem~\ref{thm:equiv},
where 
Lemma~\ref{lem:equiv_scal_h}
and~\ref{lem:2_impl_scal} are
replaced by
Lemma~\ref{lem:equiv_scal_h-2}
and~\ref{lem:2_impl_scal-2}.
Details are omitted.
We prove that
$\Dc \implies \Db$. 
By~\eqref{ineq:comp_TJ} and our assumption
there is $c=c(d,c_3)$ such that
\begin{align*}
\int_{\Rd}e^{-t\, {\rm Re}[\LCh(z)]} dz 
\leq
\int_{|z|<1/T_3} e^{-ct\, h(1/|z|)} dz 
+\int_{|z|\geq 1/T_3}e^{-t\, {\rm Re}[\LCh(z)]} dz =: I_1+I_2\,.
\end{align*}
Now, define
$$
\tilde{h}(r)=
\begin{cases}
r^{-\alpha_3} T_3^{\alpha_3} h(T_3)\, \quad &r\leq T_3\,,\\
h(r) &r>T_3\,.
\end{cases}
$$
It's not hard to verify that the function $f(r)=\tilde{h}(1/r)$
satisfies ${\rm WLSC}(\alpha_3,0,c_3/c_d)$ and therefore by
\cite[Lemma~16]{MR3165234},
\begin{align*}
I_1\leq \int_{\Rd} e^{-ct \, f(|z|)}\,dz \leq \tilde{c} \big[ f^{-1}(1/t)\big]^{d}= \tilde{c} \big[\tilde{h}^{-1}(1/t) \big]^{-d}
=\tilde{c} \big[h^{-1}(1/t) \big]^{-d}\,,\qquad t>1/h(T_3)\,.
\end{align*}
Next, for $t>2t_0$ we have
\begin{align*}
I_2=\int_{|z|\geq 1/T_3}e^{-t\, {\rm Re}[\LCh(z)]} dz
\leq  \inf_{|z|\geq 1/T_3}\left(  e^{-(t/2)  {\rm Re}[\LCh(z)]}\right) \int_{\Rd} e^{-t_0\, {\rm Re}[\LCh(z)]}\,dz\,.
\end{align*}
Since $e^{-t_0 \LCh}\in L^1(\Rd)$, then $p(t_0,x)$ exists. Thus 
by Riemann-Lebesgue lemma $e^{-t_0 \LCh}\in C_0(\Rd)$. In particular,
$\lim_{|x|\to \infty} {\rm Re}[\LCh(x)]=\infty$.
The latter implies that ${\rm Re}[\LCh(x)]\neq 0$ if $x\neq 0$
(otherwise we would have ${\rm Re}[\LCh(k x)]=0$ for some $x\neq 0$ and all $k\in\N$).
Then by continuity of $\LCh(x)$,
$$
\inf_{|z|\geq 1/T_3} \left(  e^{-(t/2)  {\rm Re}[\LCh(z)]}\right)
= \left(  e^{- \inf_{|z|\geq 1/T_3} (1/2){\rm Re}[\LCh(z)]}\right)^{t}
=c_0^t\,,\quad \mbox{where}\quad c_0\in(0,1)\,.$$
Finally, $c_0^t$ is bouded up to multiplicative constant by $[h^{-1}(1/t)]^{-d}$ (see $\Bb$). 
This ends the proof.
\end{proof}

\section{Decomposition}\label{sec:decomposition}

Let $Y$ be a L{\'e}vy process in $\Rd$
with a generating triplet $(0,\LM,\drf)$
and assume that $\Cc$ holds.
The aim of this section is to 
decompose
$Y$ into 
$Z^{1.\lambda}$ and $Z^{2.\lambda}$
is such a way that it can be used to investigate 
its density.
The idea is to some extent it is  motivated by \cite{MR1486930}.
We introduce an auxiliary L{\'e}vy measure $\nu$
satisfying for some
 $a_1\in(0,1]$, 
$$
a_1 \,\nu (dx) \leq \LM(dx)\,,
$$
and for some $a_2 \in [1,\infty)$ and all $|x|>1/T_3$,
$$
{\rm Re}[\LCh (x)] \leq a_2 \,{\rm Re}[\LCh_\nu (x)]\,. 
$$
Here
$\LCh_{\nu}$ corresponds to $(0,\nu,0)$.
We similarly write $h_{\nu}$.
For $\lambda>0$ consider the following L{\'e}vy measures
$$
\LM_{1.\lambda}(dx):=\LM(dx)-\frac{a_1}{2} \nu |_{B_{\lambda}} (dx)\,,\qquad \LM_{2.\lambda}(dx):=\frac{a_1}{2}\nu |_{B_{\lambda}} (dx)\,.
$$
We let $Z^{1.\lambda}$ and $Z^{2.\lambda}$ 
be L{\'e}vy processes with generating triplets $(0,\LM_{1.\lambda},\drf)$
and $(0,\LM_{2.\lambda},0)$, respectively. 
By analogy we write $\LCh_{1.\lambda}$, $h_{1.\lambda}$, $p_{1.\lambda}$,
$\drf^{1.\lambda}_{r}$ and $\LCh_{2.\lambda}$, $h_{2.\lambda}$, $p_{2.\lambda}$, $\drf^{2.\lambda}_{r}$. 
We collect technical inequalities that will be used without further comment.

\noindent
\begin{remark}

\noindent
(i)
 For $x\in\Rd$
$$
\frac{a_1}{2} {\rm Re}[\LCh_{\nu}(x)] \leq \frac{1}{2} {\rm Re}[\LCh(x)] \leq {\rm Re}[\LCh_{1.\lambda}(x)] \leq {\rm Re}[\LCh(x)]\,.
$$

\noindent
(ii)
For $|x|>1/T_3 $ we get
$$
a_1 \LCh_{\nu}^*(|x|)\leq \LCh^*(|x|)\leq (1/c_3) {\rm Re}[\LCh(x)] \leq (a_2/c_3) 
{\rm Re}[\LCh_\nu(x)]\leq   (a_2/c_3)\LCh_\nu^*(|x|)\,.
$$

\noindent
(iii) The characteristic exponent $\LCh_\nu$ satisfies $\Cc$ with $T_\nu=T_3$, $c_\nu=(c_3^2 a_1)/a_2 $ and $\alpha_\nu=\alpha_3$.

\noindent
(iv)
For $r>0$
$$
a_1 h_\nu(r)\leq h(r)\,,
$$
and for $r<T_3$
$$
h(r)\leq a_2 c_d h_\nu(r)\,,
$$
holds with $c_d=16(1+2d)$ by \eqref{ineq:comp_TJ}.

\end{remark}

The first result resembles
in its formulation and in the proof Lemma~\ref{lem:CI}
applied to $Z^{1.\lambda}$,
 but
it is tuned to a new approach 
and  involves auxiliary objects like
$h_\nu$.

\begin{lemma}\label{lem:first_1}
There are constants
$a_0=a_0(d,\alpha_3,c_3,a_2)\geq 1$ and
$c_{p_1}=c_{p_1}(d,\alpha_3,c_3,a_1,a_2)$
such that
for every $\lambda:=a_0 h_\nu^{-1}(1/t)<T_3$
there exists 
$|\bar{x}_t|\leq \lambda$
for which
\begin{align*}
\inf_{|y|\leq c_{p_1} \lambda } \big[ p_{1.\lambda}(t, y+ \bar{x}_t+ t \drf^{1.\lambda}_{\lambda})\big] \geq 1/(4\omega_d)\, \lambda^{-d} \,.
\end{align*}
\end{lemma}

\pf
\noindent
{\it Step 1.} 
There is a constant $a_0=a_0(d,\alpha_3,c_3,a_2)\geq 1$ such that
for 
$\lambda:=a_0 h_\nu^{-1}(1/t)<T_3$,
$$
\PP(|Z^{1.\lambda}_t- t \drf^{1.\lambda}_{\lambda}|\geq \lambda)\leq 1/2\,.
$$
Indeed, by 
\cite[page~954]{MR632968}
there is $c=c(d)$ such that for 
$r=\lambda$,
\begin{align*}
\PP(|Z_t^{1.\lambda}- t \drf^{1.\lambda}_{\lambda}|\geq r) 
&\leq c t  \left( r^{-1} \left|(\drf- \drf^{1.\lambda}_{\lambda})+\int_{\Rd} z \left(\ind_{|z|<r} - \ind_{|z|<1}\right)\LM_{1.\lambda}(dz)\right|  +h_{1.\lambda}(r)\right)\\
&= ct  h_{1.\lambda}(r)
\leq ct h(r)\,.
\end{align*}
Applying Lemma~\ref{lem:equiv_scal_h} we get
$$
h(r)\leq (c_d/c_3) a_0^{-\alpha_3} h(r /a_0)\leq a_2(c_d/c_3)^2 a_0^{-\alpha_3} h_\nu(r /a_0)=a_2(c_d/c_3)^2 a_0^{-\alpha_3} t^{-1}\,.
$$
Now, the inequality follows with $a_0=(2 c a_2 (c_d/c_3)^2 )^{1/\alpha_3}$.

\noindent
{\it Step 2.}
We note that for 
$\lambda<T_3$
there exists 
$|\bar{x}_t|\leq \lambda$ such that
\begin{align*}
p_{1.\lambda}(t, \bar{x}_t+ t \drf^{1.\lambda}_{\lambda}) \geq 1/(2\omega_d)\, \lambda^{-d}\,.
\end{align*}
It clearly follows from the continuity of $p_{1.\lambda}$ and
\begin{align*}
1/2 &\leq 1-\PP(|Z_t^{1.\lambda}- t \drf^{1.\lambda}_{\lambda}|\geq \lambda) 
= \int_{|x- t \drf^{1.\lambda}_{\lambda}|< \lambda } p_{1.\lambda}(t,x)\,dx
\leq  
\omega_d \,
\lambda^d
 \sup_{|x|< \lambda } \big[ p_{1.\lambda}(t,x+ t \drf^{1.\lambda}_{\lambda}) \big]\,.
\end{align*}

\noindent
{\it Step 3.}
We claim that
there exists a constant $c_{st3}=c_{st3}(d,\alpha_3,c_3,a_1,a_2)$
such that for every $t<1/h_\nu(T_3)$ 
we have
\begin{align*}
\sup_{x\in\Rd} |\nabla_x p_{1.\lambda}(t, x)| \leq c_{st3} /(2\omega_d) \,  \lambda^{-d-1} \,. 
\end{align*}
Since
$\LCh_\nu$ satisfies $\Cc$,
by $\Ce$ there is
$c_\nu'=c_\nu'(d,\alpha_\nu,c_\nu)$
such that for every $t<1/h_\nu(T_\nu)$,
\begin{align*}
\int_{\Rd} |z| e^{-t\, {\rm Re}[\LCh_{1.\lambda}(z)]}\, dz 
&\leq \int_{\Rd} |z| e^{-(a_1/2)t\, {\rm Re}[\LCh_\nu(z)]}\, dz 
\leq c_\nu'\left[h_\nu^{-1}(2/(a_1 t))\right]^{-d-1}\\
&\leq c_\nu'\left[ (a_1 c_\nu/(2c_d))^{1/\alpha_\nu} h_\nu^{-1}(1/t)\right]^{-d-1}\,.
\end{align*}
The last inequality follows from Lemma~\ref{lem:equiv_scal_h}.

\noindent
{\it Step 4.}
The statement of the lemma now follows. 
Indeed, by {\it Step 2.} and {\it Step 3.}  we have for every $|y|\leq 1/(2 c_{st3}) \,\lambda$,
$$
p_{1.\lambda}(t,y+ \bar{x}_t+ t \drf^{1.\lambda}_{\lambda})\geq p_{1.\lambda}(t,\bar{x}_t+ t \drf^{1.\lambda}_{\lambda})-|y| \sup_{x\in\Rd} |\nabla_x p_{1.\lambda}(t,x)|\geq 1/(4 \omega_d)\, \lambda^{-d} \,. 
$$
\qed

In what follows  we study $Z^{2.\lambda}$.

\begin{lemma}\label{lem:first_2}
Let $a_0$ be like in Lemma~\ref{lem:first_1}.
There is a constant $c_{p_2}=c_{p_2}(d,\alpha_3,c_3,a_1,a_2)\geq 1$ such that
for every $\lambda:=a_0 h_\nu^{-1}(1/t)<T_3$
and
$|x|\geq c_{p_2}\lambda^{-1}$,
$$
{\rm Re}[\LCh_\nu(x)] \leq c_{p_2} {\rm Re}[\LCh_{2.\lambda}(x)]\,.
$$
Further, $\LCh_{2.\lambda}$ satisfies $\Cc$ with $T=c_{p_2}\lambda^{-1}$, $c=c(c_3,a_2)$ and $\alpha=\alpha_3$.
\end{lemma}
\pf
{\it Step 5.}
We observe that
\begin{align*}
{\rm Re}[\LCh_\nu(x)]
&=(2/a_1) {\rm Re}[\LCh_{2.\lambda}(x)]
+\int_{|z|\geq \lambda}\big( 1-\cos(\left<x,z\right>) \big) \nu(dz)\\
&\leq (2/a_1) {\rm Re}[\LCh_{2.\lambda}(x)]+2 h_\nu(\lambda)\,.
\end{align*}
Using 
\eqref{ineq:comp_TJ} and ${\rm WLSC}$ of $\LCh_\nu^*$, for $|x| \geq  1/\lambda>1/T_\nu$ we have
\begin{align*}
2 h_\nu(\lambda)\leq c_d \LCh_\nu^*(1/\lambda)\leq (c_d/c_\nu) (|x|\lambda)^{-\alpha_\nu} \LCh^*_\nu(|x|)
\leq \left( \frac{a_2 c_d}{a_1 c_3c_\nu}\right) (|x|\lambda)^{-\alpha_\nu} 
{\rm Re}[\LCh_\nu(x)]\,.
\end{align*}
Finally, we choose $c_{p_2}$ such that
$2 h_\nu(\lambda)\leq (1/2){\rm Re}[\LCh_\nu(x)]$.
The last sentence follows from
the comparability of ${\rm Re}[\LCh_\nu(x)]$ and ${\rm Re}[\LCh_{2.\lambda}(x)]$.
\qed

In the next result we put $Z^{1.\lambda}$ and $Z^{2.\lambda}$ together to obtain estimates for the process $Y$.
Given $T\in (0,\infty]$, $a, r>0$ consider
a family of infinitely divisible probability measures,
\begin{align}\label{def:class_distr}
\mathcal{X}(T,a,r):=\{\mu\colon \mu \mbox{ is the distribution of } (Z^{2.\lambda}_t-t\drf^{2.\lambda}_\lambda) /\lambda+y \mbox{ for}&\\
 \mbox{some }
 \lambda:=a\, h_\nu^{-1}(1/t)<T
\mbox{ and } |y| \leq r &\}\,.\nonumber
\end{align}
We note that $\mathcal{X}$ is completely described by the choice of $(T,a, r)$ and $a_1, \nu$.

\begin{proposition}\label{prop:aux_lower}
Let $a_0$, $c_{p_1}$ and $\lambda$ be like in Lemma~\ref{lem:first_1}.
Take $\theta_1,\theta_2>0$
and $r_0=1+\theta_1+\theta_2$.
For all $t<1/h_\nu(T_3/a_0)$
and
$|x|\leq \theta_1 h_{\nu}^{-1}(1/t)$,
\begin{align*}
p(t,x+\Theta_t)
\geq 1/(4\omega_d) \left[ a_0  h_\nu^{-1}(1/t)\right]^{-d} \inf_{\mu \in \mathcal{X}(T_3,a_0,r_0)} \mu(B_{c_{p_1}}) \,,
\end{align*}
whenever
$\Theta_t\in\Rd$
satisfies
$|t\drf_\lambda-\Theta_t|\leq \theta_2 \lambda$
for $\lambda<T_3$.
\end{proposition}
\pf
{\it Step 6.}
Note that
$\LCh=\LCh_{1.\lambda}+\LCh_{2.\lambda}
$
and $b_{\lambda}=b^{1.\lambda}_{\lambda}+b^{2.\lambda}_{\lambda}$.
By Lemma~\ref{lem:first_1}
we have for
$\sigma_t:=x-\bar{x}_t-t \drf_\lambda+\Theta_t$,
\begin{align*}
p(t,x+\Theta_t)
&=\int_{\Rd} p_{1.\lambda}(t,x+\Theta_t-z) p_{2.\lambda}(t,z)\,dz\\
&=\int_{\Rd} p_{1.\lambda}(t,y+\bar{x}_t+ t \drf^{1.\lambda}_{\lambda}) p_{2.\lambda}(t, \sigma_t+t\drf^{2.\lambda}_{\lambda}-y)\,dy\\
&\geq \int_{|y|\leq c_{p_1}\lambda }  1/(4\omega_d)\, \lambda^{-d} \,p_{2.\lambda}(t,\sigma_t+t\drf^{2.\lambda}_{\lambda}-y)\,dy\\
&= 1/(4\omega_d)\, \lambda^{-d}\, \PP(|Z_t^{2.\lambda}-t\drf^{2.\lambda}_{\lambda}-\sigma_t|\leq c_{p_1}\lambda )\,.
\end{align*}
By Lemma~\ref{lem:first_1} and our assumptions $|\sigma_t|\leq r_0 \lambda$.
This ends the proof.
\qed

In comparison to
Lemma~\ref{lem:CI},
Proposition~\ref{prop:aux_lower}
suggests an explicit shift in the space coordinate and
gives a choice of the shift
within certain class
(see also \eqref{ineq:dif_br}).
On the other hand,
it still leaves the 
crucial question of the
positivity of
$\inf_{\mu \in \mathcal{X}(T_3,a_0,r_0)} \mu(B_{c_{p_1}})$
unresolved. 
In the next three lemmas we begin the investigation of 
$\mathcal{X}(T,a,r)$.
The issue of the positivity
is eventually addressed in Section~\ref{sebsec:low_C3}.

\begin{lemma}\label{lem:t_0}
Let $a_0$ be like in Lemma~\ref{lem:first_1}.
Then 
$\mathcal{X}(T_3,a_0,r)$ is tight
for every $r>0$.
\end{lemma}
\pf
\noindent
{\it Step 7.}
By 
\cite{MR632968}
there is $c=c(d)$ such that for 
every $\mu\in \mathcal{X}(T_3,a_0,r)$ and  $R> 1+r$,
\begin{align*}
\mu(B_R^c)
&=\PP(|(Z^{2.\lambda}_t-t\drf^{2.\lambda}_\lambda) /\lambda+y|\geq R)
\leq \PP(|(Z^{2.\lambda}_t-t\drf^{2.\lambda}_\lambda) |\geq (R-r) \lambda )\\
&\leq c t  \left( (R-r)^{-1} \lambda^{-1} \left|-\drf^{2.\lambda}_{\lambda}+\int_{\Rd} z \left(\ind_{|z|<r} - \ind_{|z|<1}\right)\LM_{2.\lambda}(dz)\right|  +h_{2.\lambda}((R-r) \lambda)\right)\\
&= ct  h_{2.\lambda}((R-r) \lambda)
= ct (a_1/2)(R-r)^{-2}\int_{|z|<\lambda} \left(|z|^2/\lambda^2 \right)\nu(dz)\\
&\leq ct \frac{(a_1/2)}{(R-r)^{2}} h_\nu(\lambda)
\leq c \frac{(a_1/2)}{(R-r)^{2}}\,,
\end{align*}
which gives the claim.
\qed

\begin{lemma}\label{lem:first_3}
Let $a_0$ be like in Lemma~\ref{lem:first_1}.
There is a constant
$c_{p_3}=c_{p_3}(d,\alpha_3,c_3,a_1,a_2)$
such that
for every $\mu \in \mathcal{X}(T_3,a_0,r)$ and $r>0$,
$$
\int_{\Rd}|\widehat{\mu}(z)|\,dz\leq c_{p_3}\,.
$$
\end{lemma}

\pf
\noindent
{\it Step 8.}
The characteristic exponent of $\mu\in \mathcal{X}$ equals $-i\left<x,y -t\drf^{2.\lambda}_\lambda/\lambda \right>+t\LCh_{2.\lambda}(x/\lambda)$.
Since
$\LCh_\nu$ satisfies $\Cc$, 
 by 
$\Cb$ there is 
$c_\nu'=c_\nu'(d,\alpha_\nu,c_\nu)$
such that for $ \lambda=a_0\, h_\nu^{-1}(1/t)<T_3$
we have
\begin{align*}
\int_{\Rd}|\widehat{\mu}(z)|\,dz&=\int_{\Rd}e^{-t {\rm Re}[\LCh_{2.\lambda}(z/\lambda)]}\,dz
= \lambda^d \int_{\Rd}e^{-t {\rm Re}[\LCh_{2.\lambda}(z)]}\,dz\\
&\leq \lambda^d \int_{|z|\leq c_{p_2}\lambda^{-1}}\,dz
+ \lambda^d \int_{\Rd}e^{-(t/c_{p_2}) {\rm Re}[\LCh_\nu(z)]}\,dz\\
&\leq \omega_d c_{p_2}^d + c_\nu' \lambda^d \left[h_\nu^{-1}(c_{p_2}/t)\right]^{-d}
\leq \omega_d c_{p_2}^d+c_\nu' a_0^d (c_{p_2}c_d/c_\nu)^{d/\alpha_\nu}\,.
\end{align*}
The last inequality follows from Lemma~\ref{lem:equiv_scal_h}.
\qed

\begin{lemma}\label{lem:first_4}
Let $a_0$ be like in Lemma~\ref{lem:first_1}.
For every $r,r_1>0$
there exists an infinitely divisible probability measure $\mu_0$ such that
$$
\inf_{\mu \in \mathcal{X}
(T_3,a_0,r)
}
\mu(B_{r_1}) \geq \mu_0(B_{r_1})\,,
$$
The measure  $\mu_0$ is a weak limit of a sequence $\mu_n \in \mathcal{X}
(T_3,a_0,r)$
and it is absolutely continuous with a continuous density
$$g_0(x)=(2\pi)^{-d}\int_{\R^d} e^{-i\left<x,z\right>} \widehat{\mu}_0(z)\,dz\,.
$$
\end{lemma}

\pf
{\it Step 9.}
Let $\mu_n$ be a sequence realizing the infimum. By 
Lemma~\ref{lem:t_0}
and
Prokhorov's theorem
we can assume that $\mu_n$ converges weakly to a probability measure $\mu_0$. Thus, since $B_r$ is open,
the inequality holds
and $\mu_0$ is infinitely divisible, see
 \cite[Theorem~8.7]{MR1739520}. 
By  
\cite[Proposition~2.5(xii) and (vi)]{MR1739520},  
Lemma~\ref{lem:first_3} and
Fatou's lemma we get
$
\int_{\Rd}|\widehat{\mu}_0(z)|\,dz \leq c_{p_3}
$.
This ends the proof.
\qed

\section{Lower bounds}\label{sebsec:low_C3}

In this section we discuss a L{\'e}vy process $Y$ in $\Rd$
with a generating triplet $(A,\LM,\drf)$.
The analysis of the upper bounds 
of transition densities 
carried out in Section~\ref{sec:equivalence}
led to lower bounds in
Lemma~\ref{lem:CI}, Corollary~\ref{lem:C6} and~\ref{cor:C7}.
As explained in
Remark~\ref{rem:sym},
Lemma~\ref{lem:CI} applied to symmetric L{\'e}vy processes gives the so called near-diagonal lower bounds. The situation becomes more complicated if the symmetry is spoiled, and an obscure shift by unknown $x_t$ appears.
This is a potential obstacle for further applications.
We propose the following correction to remove this problem: show that
at the expense of a constant one can freely choose 
$\theta>0$ for which the estimates
are valid with any $y\in\Rd$ satisfying $|y|\leq \theta h^{-1}(1/t)$.
This in turn will make it possible to remove $x_t$ by the choice of $\theta$ and $y$.
Obviously, such approach will fail in general 
even under $\Cc$,
with $\alpha$-stable subordinators as counterexamples (see Remark~\ref{rem:after_low}),
so additional restrictions will be needed.

First we 
concentrate on
 the case with non-zero Gaussian part.

\begin{lemma}\label{lem:C3gauss}
We have $\det (A)\neq 0$ if and only if $\Cc$ holds and $A \neq 0$.
If $\det(A)\neq 0$ and $\int_{\Rd}|x|^2\LM(dx)<\infty$, then $\Cc$ holds with $T_3=\infty$.
\end{lemma}
\begin{proof}
We first prove that under $\Cc$ the condition $A\neq 0$ implies $\det (A)\neq 0$. 
Indeed, if that was not the case we would have $Ax=0$ for some $|x|=1$ and then
 by \eqref{ineq:comp_TJ} with $c_d=16(1+2d)$,
\begin{align*}
c_3  h(r) r^2 
&\leq (c_d/2)  {\rm Re}[\LCh(x/r)] r^2
=(c_d/2) r^2 \int_{\Rd}
\big( 1-\cos(\left<x/r,z\right>) \big)\LM(dz)\\
& \leq c_d \int_{\Rd} \left( r^2 \land |z|^2 \right)\LM(dz)\,,
\end{align*}
which leads to a contradiction since 
the latter tends to zero
as $r\to 0^+$.
On the other hand, 
if $\det(A)\neq 0$, since
$A$ is non-negative definite,
there is $c>0$ such that $\left<x,Ax\right>\geq c |x|^2$.
We also have
$\|A\| \leq h(r)r^2  \leq h(R)R^2=:\kappa$ for $r< R$,
thus
${\rm Re}[\LCh(x)]\geq \left<x,Ax\right>\geq (c/\kappa) h(1/|x|)$
for $|x|>1/R$ and $h$ satisfies $\Aa$ with $\theta_h=R$.
Then $\Cc$ holds with $T_3=R$ by \eqref{ineq:comp_TJ}
and Lemma~\ref{lem:equiv_scal_h}.
If additionally $\int_{\Rd}|x|^2\LM(dx)<\infty$, the above is true with
$\kappa=\|A\|+\int_{\Rd}|x|^2\LM(dx)$ and $R=\infty$.
\end{proof}

Note that 
the Gaussian component
of $h$ equals $r^{-2} \|A\|$. Thus, if $A$ is non-zero,
it will dominate locally. This is reflected in the next result.

\begin{proposition}\label{thm:low_gauss}
Assume that $\Cc$ holds and  $A \neq 0$.
Then
for all $T,\theta>0$ there 
is
$\tilde{c}=\tilde{c}(d,A,\LM,T,\theta)>0$
such that
for all $0<t<T$ and $|x|\leq\theta \sqrt{t}$,
$$
p(t,x+t \drf_{\sqrt{t}}) \geq \tilde{c} \,t^{-d/2}\,.
$$
If additionally $\int_{\Rd}|x|^2\LM(dx)<\infty$, then we can take $T=\infty$ with $\tilde{c}>0$.
\end{proposition}

\pf
We consider two L{\'e}vy processes $Z^1$ and $Z^2$ 
that correspond to  $(\frac{1}{2}A,N,\drf)$ and $(\frac{1}{2}A,0,0)$, respectively.
By Lemma~\ref{lem:C3gauss} the condition $\Cc$ holds for $\LCh_1$.
Lemma~\ref{lem:CI} assures that
there is a constant 
$c=c(d,A,\LM,T)\geq 1$
such that
for every 
$t<T$ there is
$|x_t|\leq c h_1^{-1}(1/t)$
so that for every $|y|\leq (1/c) h_1^{-1}(1/t)$
we have
$
p_1(t,y+x_t+t \drf_{[h_1^{-1}(1/t)]})\geq (1/c) \left[h_1^{-1}(1/t)\right]^{-d}
$.
Since $\LCh=\LCh_1+\LCh_2$ we get
\begin{align*}
p(t, x+t \drf_{\sqrt{t}})&=\int_{\Rd} p_1(t,x+t \drf_{\sqrt{t}}- z) p_2(t,z)\,dz\\
&= \int_{\Rd} p_1(t,y +x_t + t \drf_{[h_1^{-1}(1/t)]}) p_2(t,\sigma_t-y)\,dy\\
&\geq (1/c) \left[h_1^{-1}(1/t)\right]^{-d} \PP(|Z^2_t -\sigma_t|\leq (1/c) h_1^{-1}(1/t) )\,,
\end{align*}
where $\sigma_t:=x-x_t+t\drf_{\sqrt{t}}-tb_{[h_1^{-1}(1/t)]}$.
Now,
for $r\leq R:=h_1^{-1}(1/T)$ we have
$
\frac{1}{2}\|A\| \leq h_1(r) r^2 \leq h_1(R)R^2
=:\!\kappa
$,
which by putting $r=h_1^{-1}(1/t)$, implies for $t< 1/h_1(R)=T$,
$$
1/\kappa \leq t/[h_1^{-1}(1/t)]^2\leq 2/\|A\|\,.
$$
By \eqref{ineq:dif_br}
we get for $t<T$ that
\begin{align*}
t |\drf_{\sqrt{t}} -\drf_{[h_1^{-1}(1/t)]}|\leq (1\vee \kappa)  (1\vee (2/\|A\|)^{1/2}) h_1^{-1}(1/t)
\qquad 
\mbox{and}
\qquad
|x|\leq \theta (2/\|A\|)^{1/2}\, h_1^{-1}(1/t)\,.
\end{align*}
Thus
$|\sigma_t|\leq m_1 h_1^{-1}(1/t)$
with
$m_1=m_1(d,A,\LM,T,\theta)$.
Note that
by Lemma~\ref{lem:C3gauss}
the density of $Z_t^2$ equals
$
p_2(t,x)= (2\pi t)^{-d/2} (\det (A))^{-1/2} \exp \left\{- \left<x,A^{-1}x\right>/(2t)\right\}
$.
Then
\begin{align*}
&\PP(|Z^2_t -\sigma_t|\leq (1/c) h_1^{-1}(1/t) )
=\int_{|z-\sigma_t/h_1^{-1}(1/t)|\leq 1/c} 
p_2(t/[h_1^{-1}(1/t)]^2, z)\,dz\\
&\geq \inf_{|y|\leq m_1} \int_{|z-y|\leq 1/c} \left( 2\kappa /\|A\| \right)^{-d/2} p_2(1/\kappa, z)\,dz=m_2 >0\,.
\end{align*}
Eventually, for all $t<T$ and $|x|\leq \theta \sqrt{t}$,
\begin{align*}
p(t,x+t \drf_{\sqrt{t}})
\geq 
(m_2/c) \left[h_1^{-1}(1/t)\right]^{-d} \geq   (m_2/c)  (\|A\|/2)^{1/2} \,t^{-d/2}\,.
\end{align*}
If $\int_{\Rd}|x|^2\LM(dx)<\infty$, the above is valid for all $t>0$ with $\kappa=\|A\|/2+\int_{\Rd}|x|^2\LM(dx)$.
\qed

Now we
focus on the case with
zero
Gaussian part.
We  record
that processes
satisfying
assumptions 
of Proposition~\ref{thm:low_gauss}
have a non-zero symmetric (Gaussian) part
and their trajectories
are of infinite variation
\cite[Theorem~21.9]{MR1739520}.
We exploit this two features of processes separately,
and combine them with the decomposition of
Section~\ref{sec:decomposition}
to obtain
 non-local counterparts 
of Proposition~\ref{thm:low_gauss}.
We start by 
engaging a symmetric L{\'e}vy measure~$\nu_s(dx)$.
The
assumptions and the claim
are stated by means of  
$\LCh_s$ and $h_s$
that 
correspond to the  generating triplet~$(0,\nu_s,0)$.
The result extends
 part of 
\cite[Theorem~2]{MR3357585}
and
in our setting
improves 
\cite[Theorem~2.3]{MR3139314},
\cite[Theorem~1]{MR3235175}.

\begin{theorem}\label{thm:low_jump_impr}
Assume 
that $\Cc$ holds and
 $A=0$. Suppose
there is
 $a_1\in(0,1]$ 
such that
$$
a_1 \,\nu_s (dx) \leq \LM(dx)\,,
$$
and $a_2 \in [1,\infty)$ such that for every $|x|>1/T_3$,
$$
{\rm Re}[\LCh (x)] \leq a_2 \,{\rm Re}[\LCh_s (x)]\,. 
$$
Then
for all $T, \theta>0$ there 
is
a constant
$\tilde{c}=\tilde{c}(d,\alpha_3,c_3,T_3,a_1,a_2,\nu_s,T,\theta)>0$ such that
for all $0<t<T$ and $|x|\leq\theta h_s^{-1}(1/t)$,
$$
p(t,x+t \drf_{[h_s^{-1}(1/t)]}) \geq \tilde{c} \left[ h^{-1}_s(1/t)\right]^{-d}\,.
$$
If $T_3=\infty$, then we can take $T=\infty$ with $\tilde{c}>0$.
\end{theorem}

\pf
Consider the decomposition of $Y$ introduced in Section~\ref{sec:decomposition} with $\nu=\nu_s$.
We will apply Proposition~\ref{prop:aux_lower} to conclude the statement of the theorem, but
first we prove an auxiliary result, which
complements preparatory {\it Steps 1.-9.} used in proofs of Lemmas~\ref{lem:first_1}, \ref{lem:first_2},
Proposition~\ref{prop:aux_lower}
and Lemmas~\ref{lem:first_3}, \ref{lem:t_0} and~\ref{lem:first_4}.

\noindent
{\it Step 10.}
Let  $a_0$ be taken from Lemma~\ref{lem:first_1}.
We show that
for every $r,r_1>0$,
$$\inf_{\mu \in \mathcal{X}
(T_3,a_0,r)
} \mu(B_{r_1})=c_{st10}  >0\,,$$
and $c_{st10}=c_{st10}(T_3,a_0,a_1,r,r_1,\nu_s)$.
Recall that 
$\mathcal{X}
(T,a,r)$
is defined 
in
\eqref{def:class_distr}.
Note also that $t\drf^{2.\lambda}_\lambda=0$
and $Z^{2.\lambda}_t$ is symmetric.
Let $\mu_n$, $\mu_0$ and $g_0(x)$ be like in Lemma~\ref{lem:first_4}.
Let $y_n$ be such that $\mu_n$ is the distribution of $Z^{2.\lambda}_t/\lambda+y_n$.
Since $|y_n|\leq r$, by choosing a subsequent, we can assume that $y_n$ converges to $y_0.$
Then $\tilde{\mu}_0(dx)=\mu_0(dx+y_0)$ is a symmetric infinitely divisible probability measure, as a weak limit of symmetric $\mu_n(dx+y_n)$,
with a continuous symmetric density 
$$\tilde{g}_0(x)=g_0(x+y_0)\,,$$ 
and hence
$$
\sup_{x\in\Rd} \tilde{g}_0(x)=\tilde{g}_0(0)\geq  \tilde{g}_0(x) \geq \varepsilon
\qquad \mbox{for all}\quad |x|\leq \varepsilon\,,
$$
and sufficiently small $\varepsilon>0$.
Since the support of $\tilde{\mu}_0(dx)$ is a 
group (see \cite{MR0517224} or \cite[Theorem~3]{MR1713341}), then it has to equal to $\Rd$.
Therefore $\mu_0(B_{r_2})=\tilde{\mu}_0(B_{r_2}-y_0)>0$.
This ends the proof of {\it Step 10.}

Now, the following is true.

\noindent
{\it Claim.
For every $\theta>0$ there 
are 
 $a_0=a_0(d,\alpha_3,c_3,a_2)$ and
 $\tilde{c}_1=\tilde{c}_1(d,\alpha_3,c_3,T_3,a_1,a_2,\nu_s,\theta)>0$ such that
for all $t<1/h_s(T_3/a_0)$ and $|x|\leq\theta h_s^{-1}(1/t)$,
$$
p(t,x+t \drf_{[h_s^{-1}(1/t)]}) \geq \tilde{c}_1 \left[ h^{-1}_s(1/t)\right]^{-d}\,.
$$
If $T_3=\infty$, we also have $\tilde{c}_1>0$.
}

Indeed, it holds by Proposition~\ref{prop:aux_lower}
with $\theta_1=\theta$, $\theta_2=16(1+2d)a_2$ and $\Theta_t=t \drf_{[h_s^{-1}(1/t)]}$, the application of~\eqref{ineq:dif_br}
and  {\it Step 10.}
with $r=r_0$, $r_1=c_{p_1}$.

We prove 
the final statement by extending the time horizon.
In view of the {\it Claim} we only have to consider the case $T_3<\infty$.
Let $t_0=(1/2)/ h_s(T_3/a_0)$
with $a_0=a_0(d,\alpha_3,c_3,a_2)\geq 1$ taken from the {\it Claim}.
It suffices to examine $t\in [kt_0,(k+1)t_0)$, $k\in\N$.
For $k=1$ the statement holds by
the {\it Claim}.
We show by induction that the statement is true for all $k\geq 2$.
By Chapman-Kolmogorov equation 
we have for $\bar{x}:=x+t \drf_{[h_s^{-1}(1/t)]}- t_0 \drf_{[h_s^{-1}(1/t_0)]}- (t-t_0) \drf_{[h_s^{-1}(1/(t-t_0))]}$,
\begin{align*}
&p(t, x+t \drf_{[h_s^{-1}(1/t)]})\\
&\quad \geq  \int_{|y|<h_s^{-1}(1/t_0)} p(t-t_0,y + (t-t_0) \drf_{[h_s^{-1}(1/(t-t_0))]}) \,
p(t_0, \bar{x}-y +t_0 \drf_{[h_s^{-1}(1/t_0)]})\,dy\,.
\end{align*} 
In what follows we find the upper bound of $|\bar{x}-y|$. 
By \eqref{ineq:dif_br} and $t_0\leq t-t_0$ we have
\begin{align*}
&|t \drf_{[h_s^{-1}(1/t)]}- t_0 \drf_{[h_s^{-1}(1/t_0)]}- (t-t_0) \drf_{[h_s^{-1}(1/(t-t_0))]}|\\
&=|(t-t_0)(\drf_{[h_s^{-1}(1/t)]}-\drf_{[h_s^{-1}(1/(t-t_0))]}) + t_0 ( \drf_{[h_s^{-1}(1/t)]}-\drf_{[h_s^{-1}(1/t_0)]})|\\
&\leq h_s^{-1}(1/t) \big[ (t-t_0) h(h_s^{-1}(1/(t-t_0)))+  t_0 h(h_s^{-1}(1/t_0)) \big]\\
&\leq h_s^{-1}(1/t)\, t h(h_s^{-1}(1/t_0))
\leq h_s^{-1}(1/t) ( k + 1 ) a_2(c_d/c_3)\,.
\end{align*}
We note that by Lemma~\ref{lem:equiv_scal_h} and the
comparability of $h$ and $h_s$,
$\Aa$ holds for $h_s$
with $\alpha_{h_s}=\alpha_3$, $\theta_{h_s}=T_3$ and $C_{h_s}= a_2(c_d/c_3)^2/a_1$.
We extend this scaling as in Remark~\ref{rem:rozciaganie} using
$R:=h_s^{-1}(1/[(k+1)t_0])$.
Then $\Aa$ holds for $h_s$ with 
$\alpha_{h_s}=\alpha_3$, $\tilde{\theta}_{h_s}=R$ and $\widetilde{C}_{h_s}$ (resulting from the extension). In particuliar,
$1/t > h_s(\tilde{\theta}_{h_s})$
and by Lemma~\ref{lem:equiv_scal_h},
$$
h_s^{-1}(1/t)\leq (\widetilde{C}_{h_s} t/t_0 )^{1/\alpha_{h_s}} h_s^{-1}(1/t_0)
\leq ((k+1) \widetilde{C}_{h_s})^{1/\alpha_{h_s}} h_s^{-1}(1/t_0)\,.
$$
Therefore 
$|\bar{x}-y|\leq \theta_1 h_s^{-1}(1/t_0)$, 
where
$\theta_1=\theta_1(d,\alpha_3,c_3,T_3,a_1,a_2,\nu_s,k,\theta)$.
Then by
the {\it Claim},
$$
p(t_0, \bar{x}-y +t_0 \drf_{[h_s^{-1}(1/t_0)]})\geq \tilde{c}_1 \left[ h^{-1}_s(1/t_0)\right]^{-d}.
$$
Since $t-t_0 \in [(k-1)t_0,k t_0)$ and $|y|<h_s^{-1}(1/t_0)\leq h_s^{-1}(1/(t-t_0))$, by the induction hypothesis,
$$
p(t-t_0,y + (t-t_0) \drf_{[h_s^{-1}(1/(t-t_0))]})\geq \tilde{c}_{k-1} \left[ h^{-1}_s(1/(t-t_0))\right]^{-d}.
$$
Finally,
$$
p(t, x+t \drf_{[h_s^{-1}(1/t)]})\geq   \tilde{c}_1 \omega_d \tilde{c}_{k-1}  \left[ h^{-1}_s(1/(t-t_0))\right]^{-d}
\geq \tilde{c}_k  \left[ h^{-1}_s(1/t)\right]^{-d}.
$$

\qed

\begin{theorem}\label{thm:low_jump_impr_1}
Assume 
that $\Cc$ holds with $\alpha_3\geq 1$ and
 $A=0$.
 Then
for all $T, \theta>0$ there 
is
a constant
$\tilde{c}=\tilde{c}(d,\alpha_3,c_3,T_3,\LM,T,\theta)>0$ such that
for all $0<t<T$ and $|x|\leq\theta h^{-1}(1/t)$,
$$
p(t,x+t \drf_{[h^{-1}(1/t)]}) \geq \tilde{c} \left[ h^{-1}(1/t)\right]^{-d}\,.
$$
If $T_3=\infty$, then we can take $T=\infty$ with $\tilde{c}>0$.
\end{theorem}

\pf
Consider the decomposition of $Y$ introduced in Section~\ref{sec:decomposition} with $\nu=\LM$ and $a_1=a_2=1$.
Then the proof is the same as that of
Theorem~\ref{thm:low_jump_impr},
only the justification of {\it Step 10.} is different,
because instead of using the symmetry of $\nu$ we take advantage of
the assumption that $\alpha_3\geq 1$.

\noindent
{\it Step 10.}
Let  $a_0$ be taken from Lemma~\ref{lem:first_1}.
We show that for every $r,r_1>0$,
$$\inf_{\mu \in \mathcal{X}(T_3,a_0,r)
} \mu(B_{r_1})=c_{st10}  >0\,,$$
with $c_{st10}=c_{st10}(T_3,a_0,r,r_1)$.
Let $\mu_n$, $\mu_0$ and $g_0(x)$ be like in Lemma~\ref{lem:first_4}.
We denote by
$\LCh_n(x)$ and
$\LCh_0(x)$ the characteristic exponents
corresponding to $\mu_n$ and $\mu_0$.
By \cite[(8.11)]{MR1739520}
we have that
${\rm Re}[\LCh_n(x)]$
converges to ${\rm Re}[\LCh_0(x)]$ 
and
$\LCh_n^*$ converges to $\LCh_0^*$.
Since
${\rm Re}[\LCh_n(x)]= t {\rm Re} [\LCh_{2.\lambda}(x/\lambda)]$
and 
$\LCh_n^*(r)= t \LCh_{2.\lambda}^*(r/\lambda)$,
by Lemma~\ref{lem:first_2} we get 
that
$\Cc$ holds for $\LCh_0$
with $T_0=c_{p_2}$, $c_0=c_0(c_3,a_2)$
and $\alpha_0=\alpha_3\geq 1$.
If it happens that $\LCh_0$
has non-zero Gaussian part, then
Lemma~\ref{lem:C3gauss}
guarantees that
the support of the measure $\mu_0$ equals $\Rd$,
which ends the proof in that case.
 Suppose that $\LCh_0$ has zero Gaussian part
and denote by $\LM_0(dz)$ the corresponding L{\'e}vy measure.
We will justify that for every $x\in\Rd$, $x\neq 0$,
\begin{align}\label{eq:infty_x}
\int_{|z|<1}|\!\left<x,z\right>\!|\,\LM_0 (dz)=\infty\,.
\end{align}
Let $\Pi_1$ be a projection on a subspace spanned by $x/|x|$.
Then
\begin{align*}
\int_{|z|<1}|\!\left<x/|x|,z\right>\!|\,\LM_0(dz)
\geq  \int_{|\Pi_1 z|<1} 
|\Pi_1 z|\,\LM_0 (dz)
-\LM_0 (B_1^c)
=\int_{|z|<1} 
|z|\,\LM_1(dz) -\LM_0 (B_1^c)\,,
\end{align*}
where $\LM_1(dz)$
is a L{\'e}vy measure of 
an infinitely divisible distribution 
that is the $\Pi_1$ projection of $\mu_0$ (see \cite[Proposition~11.10]{MR1739520}).
We denote by $h_1$
the concentration function for $\LM_1(dz)$.
By 
$\Cc$ for $\LCh_0$ and
Lemma~\ref{lem:C8} 
we get $\Aa$ for $h_1$ with $\alpha_1\geq 1$. Then
\eqref{eq:infty_x} follows from
Lemma~\ref{lem:int_infty}.
Finally,
by
\cite[Corollary on page 232]{MR923491}
or
\cite[Theorem~3]{MR1713341}
the support of $\mu_0$ is $\Rd$. This ends the proof.
\qed

\begin{remark}\label{rem:after_low}

\noindent
(i)
One of the main improvements of Theorem~\ref{thm:low_jump_impr}
and~\ref{thm:low_jump_impr_1}
in comparison to known results is that
we can arbitrarily choose $\theta>0$. 
We take advantage of that in Proposition~\ref{prop:cone}.

\noindent
(ii)
The assumption $a_1 \nu_s(dz)\leq \LM(dz)$ 
of Theorem~\ref{thm:low_jump_impr}
cannot by replaced by a weaker condition $a_1 {\rm Re}[\LCh_s (x)] \leq {\rm Re}[\LCh (x)]$, because the latter and other assumptions of the theorem are satisfied for $\alpha$-stable subordinators (take $\LCh_s$ to be 
the characteristic exponent of the isotropic $\alpha$-stable process), but
the statement is not true for that process. Namely,
if $\theta>0$ is large enough, then
$p(t,x+t \drf_{[h_s^{-1}(1/t)]})=0$ for some $0<t<T$ and $x\in\RR$ satisfying 
$|x|\leq\theta h_s^{-1}(1/t)$.

\noindent
(iii)
The assumption ${\rm Re}[\LCh (x)] \leq a_2 \,{\rm Re}[\LCh_s (x)]$ 
of Theorem~\ref{thm:low_jump_impr}
 holds if a stronger condition $\LM(dz)\leq a_2\, \nu_s(dz)$ is satisfied, but the latter is much more restrictive (see also Example~\ref{ex:stable_little_sym}).

\end{remark}

\section{Examples and applications}\label{sec:ex_ap}

We apply Theorem~\ref{thm:low_jump_impr} to a L{\'e}vy process $Y$ in $\Rd$ which is the sum of the (symmetric)
cylindrical $\alpha$-stable process and any arbitrarily chosen independent $\alpha$-stable process $\alpha\in (0,2)$.
\begin{example}\label{ex:stable_little_sym}
Let $\drf\in\Rd$ and 
define
$$
\LM(dz)= \nu_s(dz)+\nu_a(dz)\,,
$$
where for $\alpha \in (0,2)$,
$$
\nu_s(dz)=\mathcal{A}_{\alpha}\sum_{k=1}^d |z_k|^{-1-\alpha}dz_k \prod_{\substack{i=1\\ i\neq k}}^d
\delta_{\{0\}}(d z_i)\,, \qquad z=(z_1,\ldots,z_d)\,,
$$
and
\begin{align}\label{def:stable_arbitrary}
\nu_a(B) \approx \int_{S} \lambda(d\xi) \int_0^{\infty}\ind_{B}(r\xi) \frac{dr}{r^{1+\alpha}}\,,\qquad B\in \mathcal{B}(\Rd)\,.
\end{align}
Here 
$\mathcal{A}_{\alpha}=2^{\alpha}\Gamma((1+\alpha)/2)/(\pi^{1/2}|\Gamma(-\alpha/2)|)$,
$S=\{x\in\Rd \colon |x|=1\}$
and $\lambda$ is a finite measure on~$S$.
Then
Theorem~\ref{thm:low_jump_impr} applies
 to a L{\'e}vy process $Y$ with the generating triplet $(0,\LM,b)$.
Indeed, first note that $\nu_s$ is a special case of $\nu_a$ with $\lambda$
having properly chosen atoms on 
the sphere
and that 
\begin{align}\label{eq:stable_h}
h_a(r) \approx  r^{-\alpha} \lambda(S),\qquad r>0\,.
\end{align}
Therefore, by $\nu_s(dz)\leq \LM(dz)$ and \eqref{ineq:comp_TJ} we get
$$
d^{-\alpha/2} |x|^{\alpha}\leq |x_1|^{\alpha}+\ldots+|x_d|^{\alpha} = {\rm Re}[\LCh_s (x)] \leq {\rm Re}[\LCh (x)]\leq \LCh^* (x) \leq 2 h(1/|x|)\leq c \,|x|^{\alpha}\,,
$$
for $c$ that depends only on $\alpha$ and $\lambda$.
This shows that 
the assumptions of Theorem~\ref{thm:low_jump_impr}
are satisfied. In particular $\Cc$ holds
and $T_3=\infty$.
We emphasize that
for such $\LM$ one can rarely expect to have
 $\LM(dz)\leq c \,\nu_s(dz)$ 
for some constant $c$. 
 The latter as an assumption
 would dramatically reduce admissible measures $\lambda$.
\end{example}

It has been announced in the introduction
that any $\alpha$-stable processes $\alpha\in (0,2)$ in one dimension satisfies $\Cc$. It follows from
Remark~\ref{rem:d=1}
and
\eqref{eq:stable_h}.

\begin{example}
Let $d=1$ and $Y$ be a L{\'e}vy process with the generating triplet $(0,\LM,0)$, where
$$
\LM(dx)=  |x|^{-2} \ind_{x< 0} \,dx\,.
$$
Note that $\LM(dx)$ is of the form \eqref{def:stable_arbitrary}
with $\alpha=1$ and $\lambda(d\xi)=\delta_{\{-1\}}(d\xi)$, i.e.,
$Y$ is a (one-sided) $1$-stable process.
Then
\begin{align*}
\PP(Y_t\in (-\infty,0))\longrightarrow 0\,, \quad \mbox{as}\quad t\to 0^+\,.
\end{align*}
Indeed, using the notation of 
\cite[Theorem~1]{MR2041833}
we have $M(x)=T(x)=-D(x)=x^{-1}$, $A(x)=-1-\ln(x)$ and $U(x)=2x$. Thus $A(x)/\sqrt{U(x)M(x)} \to +\infty$ as $x\to 0^+$.
\end{example}

The above example explains a restriction to $\alpha_3>1$ in the following result.

\begin{proposition}\label{prop:cone}
Assume 
that $\Cc$ holds with $\alpha_3> 1$ and
 $A=0$.
For $\lambda>0$ let 
$$
C_{\lambda}=\{x\in\Rd\colon x_d > \lambda |\tilde{x}| \,,\, \tilde{x}=(x_1,\ldots,x_{d-1},0)\}.
$$
For every $T>0$ there is 
 a constant
 $c=c(d,\alpha_3,c_3,T_3,N,T,|b|)$
 such that for every
orthogonal matrix $O$
and
for all $t<T$,
\begin{align*}
 \PP(X_t\in O C_{\lambda}) \geq c>0 \,.
\end{align*}
\end{proposition}
\pf
By Remark~\ref{rem:rozciaganie}
and Corollary~\ref{cor:drf_r_drf} 
there is $\theta_1=\theta_1(d,\alpha_3,c_3,T_3,h,T)$ such that
$
t |\drf_{[h^{-1}(1/t)]}-b|
\leq \theta_1 h^{-1}(1/t)
$ for all $t<T$.
Using Remark~\ref{rem:rozciaganie} and $\Ab$ we also get for
$\theta_2=\theta_2(c_3,T_3,h,T,|b|)$ and all $t<T$,
 that 
$
|tb| \leq \theta_2 \,h^{-1}(1/t)
$.
Let $|x|\leq h^{-1}(1/t)$.
Then $\bar{x}=x-t \drf_{[h^{-1}(1/t)]}$ satisfies
$|\bar{x}|\leq \theta h^{-1}(1/t)$ for all $t<T$ with 
$\theta=\theta_1+\theta_2$.
By
Theorem~\ref{thm:low_jump_impr_1} we have
$$
p(t,x)=p(t,\bar{x}+t \drf_{[h^{-1}(1/t)]})\geq \tilde{c} \left[h^{-1}(1/t)\right]^{-d}.
$$
Finally,
\begin{align*}
 \PP(X_t\in OC_{\lambda}) \geq \int_{O C_{\lambda}\cap B_{h^{-1}(1/t)}} p(t,x) \,dx
\geq  \tilde{c}
\left[h^{-1}(1/t)\right]^{-d} |OC_{\lambda} \cap B_{h^{-1}(1/t)}|
= c>0\,.
\end{align*}
\qed

Define the firs exit time from an open set $D$ by
$$\tau_D=\inf\{t>0:X_t\in D^c\}.$$
\begin{corollary}Assume 
that $\Cc$ holds with $\alpha_3> 1$. Let 
an open and bounded set $D\subset \Rd$ have the outer cone property. Then
every point from $D^c$ is regular for $D$, i.e., $\PP^x(\tau_D=0)=1$ for every $x\in D^c$.
\end{corollary}
\pf
By the right continuity of paths $X_t$ we may and do assume that $x\in\partial D$.
For every $t>0$, 
$$\PP^x(\tau_D\leq t)\geq \PP^x(X_t\in D^c).$$
By the outer cone property and Proposition \ref{prop:cone} we get
$$\PP^x(\tau_D\leq t)\geq c,\quad t<T.$$
This implies that 
$$\PP^x(\tau_D=0)\geq c>0.$$
Applying Blumenthal's $0-1$ law ends the proof.
\qed

\vspace{.1in}


\small

\end{document}